\documentclass[10pt]{amsart}

\usepackage{amsmath}
\usepackage{dsfont}
\usepackage{mathrsfs}
\usepackage{amssymb}
\usepackage{hanging}
\usepackage{latexsym,amsfonts,amsthm,geometry}

\usepackage{color}
\usepackage{tikz}

\allowdisplaybreaks[3]

\usepackage{multirow}

\usepackage{enumitem}

\usepackage{bm}
\newtheorem{theorem}{Theorem}[section]
\newtheorem{lemma}[theorem]{Lemma}

\newtheorem{corollary}[theorem]{Corollary}
\newtheorem{proposition}[theorem]{Proposition}
\theoremstyle{definition}
\newtheorem{definition}[theorem]{Definition}
\theoremstyle{remark}
\newtheorem{remark}[theorem]{Remark}

\newcommand{\bn}{{\bm n}}
\newcommand{\bt}{{\bm t}}
\newcommand{\bu}{{\bm u}}
\newcommand{\bv}{{\bm v}}
\newcommand{\bw}{{\bm w}}
\newcommand{\bp}{{\bm p}}
\newcommand{\bbR}{\mathbb{R}}
\newcommand{\bcurl}{{\bf curl}}
\newcommand{\p}{\partial}
\newcommand{\nab}{\nabla}
\newcommand{\bxi}{{\bm \xi}}
\newcommand{\bmu}{{\bm \mu}}
\newcommand{\calT}{\mathcal{T}}
\newcommand{\calE}{\mathcal{E}}
\newcommand{\uD}{\underline{D}}
\newcommand{\hessone}{\underline{D}^2}
\newcommand{\hesstwo}{\nab_\Gamma^2}
\newcommand{\Grad}{\nabla\!\!\!\!\nabla}
\newcommand{\bbP}{\mathbb{P}}
\newcommand{\bq}{{\bm q}}
\newcommand{\bH}{{\bm H}}
\newcommand{\calP}{\mathcal{P}}

\usepackage{lipsum}

\newcommand{\br}{{\bm r}}

\def\avg#1{\{\hspace{-3pt}\{#1\}\hspace{-3pt}\}}
\def\jump#1{[\hspace{-2pt}[#1]\hspace{-2pt}]}
\def\tbar#1{|\hspace{-2pt}|\hspace{-2pt}|#1|\hspace{-2pt}|\hspace{-2pt}|}

\numberwithin{equation}{section}

\begin{document}

\parindent=0in

\title[$C^0$ IP method for surface stream function]{
A $C^0$ interior penalty method for the stream function formulation
of the surface Stokes problem}

\author[M.~Neilan \& H.~Wan]{Michael Neilan \email{\lowercase{neilan@pitt.edu}} \and Hongzhi Wan\email{\lowercase{HOW41@pitt.edu}}}
\thanks{Supported in part by the NSF grant DMS-2309425.}
\date{}
\maketitle

\begin{abstract}
We propose a $C^0$ interior penalty method
for the fourth-order stream function formulation of
the surface Stokes problem. The scheme utilizes
continuous, piecewise polynomial spaces defined
on an approximate surface. We show
that the resulting discretization is positive
definite and derive error estimates
in various norms in terms of the polynomial
degree of the finite element space as well
as the polynomial degree to define the geometry approximation.
A notable feature of the scheme is that it does not explicitly 
depend on the Gauss curvature of the surface. This is achieved
via a novel integration-by-parts formula for the surface biharmonic operator.
\end{abstract}

\thispagestyle{empty}

\section{Introduction}
Let $\Gamma$ be a smooth, simply connected compact oriented 
hypersurface in $\bbR^3$ without boundary and with outward unit normal $\bn$. 
The surface Stokes problem seeks the fluid velocity  $\bu: \Gamma \rightarrow \bbR^3$ with $\bu \cdot \bn = 0$ and the surface fluid pressure $p:\Gamma \to \bbR$ such that
\begin{subequations} \label{eqn:surfaceStokes}
\begin{align} \label{eqn:momentum}
- {\bf P}\,{\rm div}_\Gamma(E_\Gamma(\bu)) + \bu + \nabla_\Gamma p
\begin{aligned}
& = {\bm f} \quad & \text{on } \Gamma,
\end{aligned} \\
{\rm div}_\Gamma\bu
\begin{aligned}
& = 0 \quad & \text{on } \Gamma,
\end{aligned}
\end{align}
\end{subequations}
where ${\bm f} \in L^2(\Gamma)^3$ with ${\bm f} \cdot \bn  = 0$ is a given force vector,
$E_\Gamma(\bu)$ is the deformation tensor, and ${\bf P}$ is the tangential projection operator.
Further details and notation are given in the next section.
The zeroth-order term is included
in \eqref{eqn:momentum} to ensure uniqueness 
of the velocity solution and to avoid technicalities 
related to Killing fields, i.e., non-trivial tangential
vector fields in the kernel of the deformation tensor $E_\Gamma$ \cite{BDL20}.

Surface (Navier)-Stokes equations arise in various application models
including emulsion foams and biological membranes \cite{Scriven60,SlatteryEtal07}, computer graphics \cite{Elcott_stable},
and geophysics \cite{VoigtEtal17,Sasaki15}. As such there has been recent interest in developing
finite element methods for incompressible fluids posed on surfaces.
A natural choice are surface finite element methods (SFEMs)
based on stable Euclidean $\bH^1\times L^2$ conforming velocity-pressure pairs. 
In this approach, finite element spaces are defined on a discrete approximate 
surface through polynomial mappings, and the tangential velocity constraint 
is enforced weakly via penalization or Lagrange multipliers \cite{Fries18,JankuhnEtal18,BradnerEtal22,Hardering23}. 
While popular and relatively straightforward to implement,
this approach requires excess degrees
of freedom, as the velocity is approximated by vectors 
in all of $\bbR^3$, instead of tangential vectors. In addition,
these methods may require a superparametric approximation of
the outward unit normal of the surface to guarantee optimal-order convergence.
Recently it was shown in \cite{HarderingP23} 
that the tangential component of the solution converges
with optimal order using standard isoparametric geometry approximations;
however, $L^2$ errors are suboptimal if affine surface approximations are used.

Alternatively,
non-conforming velocity-pressure pairs
have been proposed and analyzed in \cite{BDL20,LedererEtal20},
where the velocity is approximated using exactly tangential and 
Piola-mapped $\bH(div)$-conforming BDM spaces.
These schemes are analogous to the Euclidean finite element methods in \cite{CockburnEtal07}.
The velocity spaces in these schemes are exactly tangential to the 
discrete surface, but due to their nonconformity, additional edge-integral terms are added to ensure their consistency
and stability. An exception is the recent work \cite{DemlowNeilan24}, where
a $\bH(div)$-conforming space, 
based on the lowest-order Euclidean Mini element, is constructed
that possess sufficient weak-continuity properties to guarantee convergence without including edge-integral terms.

In contrast to SFEM, Trace FEM is a discretization technique for surface 
PDEs based on a background (bulk) mesh in $\bbR^3$.
In this framework, finite element spaces are defined on the three-dimensional mesh,
and the traces of such function are used as the approximating space.
To ensure stability of the resulting scheme and algebraic system,
penalty terms are included in their formulation. 
While quite advantageous for dynamic and coupled fluid computations, 
this approach requires extraneous DOFs, as the spaces are defined on a 3D geometry.

Another discretization technique, and the focus of this paper,
is based on the surface stream function formulation.
If $\Gamma$ is simply connected,
then there exists a unique stream function $\phi\in H^2(\Gamma)$
with $\int_\Gamma \phi = 0$ such 
that $\bu = {\bf{curl}}_{\Gamma} \phi$.
Formally substituting this expression into \eqref{eqn:momentum}
and taking the curl of the resulting equation yields
\begin{equation}\label{eqn:streamfunction}
\frac12 \Delta_{\Gamma}^2 \phi + {\rm div}_{\Gamma} ((K-1) \nab_{\Gamma} \phi) = -{\rm curl}_{\Gamma}{\bm f},
\end{equation}
where $K$ is the Gauss curvature of $\Gamma$.
SFEM and TraceFEM discretizations for \eqref{eqn:streamfunction}
based on a Ciarlet-Raviart mixed formulation (cf.~\cite{CiarletRaviart74}),
are proposed and analyzed in
\cite{Voight12,Reusken20,BrandnerReusken20,BradnerEtal22}.
This approach  introduces an auxiliary unknown $\psi = \Delta_\Gamma \phi$,
and utilizes surface Lagrange finite element spaces in its discretization.
Thus, while the numerical method is relatively simple to implement,
and is supported in current finite element software, the mixed formulation
results in a relatively large number of unknowns and  a saddle-point structure.
In addition, such mixed methods require the computation of an approximate
Gauss curvature, which may require higher-order approximations
of the surface outward normal \cite{BrandnerReusken20,BradnerEtal22}.
Finally, the extension of the Ciarlet-Raviart discretization technique towards
the Stokes problem on surfaces with boundary is less clear, as boundary
conditions for the auxiliary variable $\psi$ may not be explicitly given.

In this paper, we propose and analyze a SFEM method for the surface Stokes problem
\eqref{eqn:momentum} based on a primal $C^0$ interior penalty (IP) method applied 
towards the  formulation \eqref{eqn:streamfunction}.   
Similar to $C^0$ IP methods in the Euclidean case \cite{EngelEtal02,BrennerSung05},
the proposed scheme uses continuous, piecewise polynomial spaces
as the approximation spaces, and consistency 
and symmetry is enforced by interelement contributions
of jumps and averages. This construction is done
on an approximate surface $\Gamma_h$ defined via a polynomial-mapped approximation of $\Gamma$.
A notable feature of the proposed scheme is that it does not require an explicit approximation
of the Gauss curvature.
This is achieved through the use of a novel surface Hessian-type operator 
and integration-by-parts formulas. As a result, 
coercivity and continuity 
properties of the proposed method mostly follow the same arguments as its Euclidean counterpart.
We show that the method converges and derive
explicit error estimates with respect to both the finite element degree $k$
and the degree of the geometric approximation $k_g$.

One advantage of the $C^0$ IP method
is that the discretization represents a positive definite
system involving a single unknown.
In addition, the discrete velocity
is recovered by simply taking the tangential curl
of the discrete stream function. This numerical velocity
is thus exactly tangential to the (approximate) surface
and does not need any ad hoc penalization techniques to enforce this constraint.
Potential  disadvantages of the stream function approach (compared to a velocity-pressure-based formulation) is it requires $\Gamma$ to be simply connected. Furthermore, as the stream function formulation
is fourth-order, the condition number of any  (primal) discretization is expected to scale like $O(h^{-4})$, where $h$ is the mesh width.

Recently, several finite element methods
for the surface biharmonic operator have been
proposed and analyzed. Closely related
to the present work is \cite{LarssonLarson17},
where the authors propose and analyze the lowest-order $C^0$ IP method for the surface 
biharmonic problem on affine approximations.  A variation 
of this approach, utilizing a surface gradient recovery operator, is presented in \cite{ZhangEtal24}.
These discretizations are based on standard integration-by-parts formulas
of the surface biharmonic operator, leading to bilinear forms involving
the products of Laplacians.  Such approaches can be easily formulated
towards the fourth-order problem \eqref{eqn:streamfunction}, although the stability
and convergence of the scheme is unclear due to the indefinite 
low-order terms in the PDE.  This approach also leads to schemes where approximate Gauss
curvatures are required, which may require higher-order geometry approximations
or additional computational resources \cite{BrandnerReusken20,BradnerEtal22,Walker24}.
The approach we take in this manuscript circumvents these issues
through an integration-by-parts procedure.

The rest of the paper is organized as follows.
In the next section, we set the notation and provide
integration-by-parts identities for the surface biharmonic operator.
In Section \ref{sec-Mesh}, we define the discrete surface
approximation $\Gamma_h$ and mappings between $\Gamma$ and $\Gamma_h$.
Section \ref{sec-Method} states the $C^0$ IP method and proves
the continuity and coercivity of the corresponding bilienar form.
In Section \ref{sec-Geo}, we derive estimates of the geometric
inconsistencies of the scheme, and in Section \ref{sec-converge}
we prove error estimates of the $C^0$ IP method in a discrete $H^2$
norm and $H^m$ norms ($m=0,1$). Numerical experiments are 
given in Section \ref{sec-numerics}. Finally, some of the technical proofs
in the paper are provided in the appendix.

\section{Preliminaries}

Let $\Gamma$ be a smooth, simply connected compact 
oriented hypersurface in $\bbR^3$ without boundary.
We denote by $U_\delta$ a $\delta$-neighborhood of
$\Gamma$ with $\delta>0$ sufficiently small such that the
signed distance function $d$ is well-defined in $U_\delta$
(with $d<0$ in the  interior of $\Gamma$). Set $\bn = \nab d$ to be 
the outward unit normal of $\Gamma$, extended to $U_\delta$, where the gradient is understood as a column vector. 
The Weingarten map is ${\bf H} = D^2 d$, the Hessian
matrix of $d$.
The tangential projection operator is given by
\begin{align*}
{\bf P} & = {\bf I} - \bn\otimes \bn.
\end{align*}
The closest point projection is 
\[
\bp(x) = x - d(x) \bn(x).
\]

Given a scalar function $\psi: \Gamma \to \bbR$,
we set its extension $\psi^e:U_\delta \to \bbR$ as $\psi^e(x) = \psi(\bp(x))$
for $x\in U_\delta$. This definition is extended
to vector fields component wise (so that $({\bm\psi}^e)_j = \psi_j^e$).
The surface gradient of $\psi$ is 
\begin{align}\label{eqn:surfaceGrad}
\nabla_{\Gamma}\psi & = {\bf P}\nabla\psi^e= 
\begin{bmatrix}
\underline{D}_{1}\psi \\
\underline{D}_{2}\psi \\
\underline{D}_{3}\psi
\end{bmatrix},
\end{align}
i.e., $\underline{D}_j \psi$ is the $j$th component of
$\nab_\Gamma \psi$.
 The surface curl operator of $\psi$
is 
\[
\bcurl_{\Gamma}\psi  = \bn \times \nabla_{\Gamma}\psi.
\]
For a vector-valued function $\bv = [v_1,v_2,v_3]^\intercal$, its
Jacobian $\Grad \bv$ satisfies $(\Grad \bv)_{i,j} = \frac{\p v_i}{\p x_j}$ for $i,j=1,2,3$.
The surface Jacobian and surface deformation tensor are defined, respectively, as
\begin{equation}
\label{eqn:surfaceJacobian}
\begin{split}
\Grad_\Gamma \bv & = {\bf P}\Grad\bv^e {\bf P}, \\
E_{\Gamma}(\bv) &= \frac{1}{2}\left(\Grad_{\Gamma}\bv + \Grad_{\Gamma}\bv^\intercal\right).
\end{split}
\end{equation}
The surface divergence operator is defined as the trace of the Jacobian, ${\rm div}_{\Gamma}\bv = {\rm tr}(\Grad_\Gamma \bv)$, and the scalar curl operator is ${\rm curl}_{\Gamma}\bv = {\rm div}_\Gamma (\bv\times \bn)$.

Next, we provide notions of several second-order operators
 that will be used throughout the paper.
\begin{definition} \label{def:2.1}
We define the nonsymmetric surface Hessian $\hessone_\Gamma \psi:\Gamma\to \bbR^{3\times 3}$ of a scalar function $\psi$ such that
 $(\hessone_\Gamma \psi)_{i,j} = \uD_j \uD_i \psi$ for $i,j=1,2,3$.
 The projected surface Hessian $\hesstwo \psi:\Gamma\to \bbR^{3\times 3}$
 satisfies $\left(\nabla_{\Gamma}^{2} \psi\right)_{i,j} = \left(\Grad_{\Gamma}\nabla_{\Gamma} \psi\right)_{i,j}$
 for $i,j=1,2,3.$ Finally, we define the second-order (Hessian-like) operator
\[
H_\Gamma(\psi):= E_\Gamma(\bcurl_\Gamma \psi),
\]
where we recall $E_\Gamma(\cdot)$ is the surface deformation tensor given by \eqref{eqn:surfaceJacobian}.
The Laplace-Beltrami operator is the trace for either $\hessone(\cdot)$ or $\hesstwo(\cdot)$, i.e.,
\[
\Delta_\Gamma \psi = {\rm div}_{\Gamma} \nab_{\Gamma} \psi
= {\rm tr}(\nab_\Gamma^2 \psi) = {\rm tr}(\underline{D}^2_\Gamma \psi) = \sum_{i=1}^3 \underline{D}_i^2 \psi.
\]
\end{definition}

\begin{remark}
Note that $\nab_{\Gamma}^2 \psi$ is symmetric (cf.~Lemma \ref{lem:HessIdentities} and \cite{Delfour00}),
whereas $\underline{D}^2_\Gamma \psi$ is not. In particular, there holds \cite[Lemma 2.6]{DziukElliott13}
\begin{equation}\label{eqn:NotSymmetric}
(\underline{D}^2_\Gamma \psi)^\intercal - \underline{D}^2_\Gamma \psi
= \bn \otimes ({\bf H}\nab_\Gamma \psi) - ({\bf H}\nab_\Gamma \psi)\otimes \bn.
\end{equation}
\end{remark}

The next two lemmas provide explicit relationships
between the three notions of the surface Hessian operators.

\begin{lemma}\label{lem:HessIdentities}
    There holds for $\psi:\Gamma \to \bbR$,
\begin{equation}\label{nHessIdentity}
\nabla_{\Gamma}^{2}\psi = {\bf P}\underline{D}^{2}\psi 
= {\bf P}\nabla^{2}\psi^e {\bf P}.
\end{equation}
\end{lemma}
\begin{proof}
To prove the first inequality in \eqref{nHessIdentity}, we first 
consider a vector-valued function
$\bv = [v_1,v_2,v_3]^\intercal: \Gamma \rightarrow \mathbb{R}^{3}$. Then
since  ${\bf P}^\intercal = {\bf P}$,  the $i$th row of $\Grad \bv^e {\bf P}$ is
\[
(\Grad \bv^e {\bf P})_{i,:} = (\nabla v^e_{i})^{\intercal}{\bf P} = \left({\bf P}\nabla v^e_{i}\right)^\intercal = \nabla_{\Gamma}v_{i}^\intercal,
\]
which implies
\[
\Grad_{\Gamma}\bv= {\bf P}\Grad\bv^e {\bf P} = {\bf P}
\begin{bmatrix}
\nabla_{\Gamma}v_{1}^{\intercal} \\
\nabla_{\Gamma}v_{2}^{\intercal} \\
\nabla_{\Gamma}v_{3}^{\intercal}
\end{bmatrix}.
\]
Setting $\bv = \nabla_{\Gamma}\psi$ proves
the first equality in \eqref{nHessIdentity}:
\[
\nabla_{\Gamma}^{2}\psi = {\bf P}
\begin{bmatrix}
\underline{D}_{1}\underline{D}_{1}\psi & \underline{D}_{2}\underline{D}_{1}\psi & \underline{D}_{3}\underline{D}_{1}\psi \\
\underline{D}_{1}\underline{D}_{2}\psi & \underline{D}_{2}\underline{D}_{2}\psi & \underline{D}_{3}\underline{D}_{2}\psi \\
\underline{D}_{1}\underline{D}_{3}\psi & \underline{D}_{2}\underline{D}_{3}\psi & \underline{D}_{3}\underline{D}_{3}\psi
\end{bmatrix}
= {\bf P}\underline{D}^{2}\psi.
\]

To prove the second equality in \eqref{nHessIdentity},
we use the product rules
\begin{align*}
\Grad(f{\bm g}) & = {\bm g}\nabla f^{\intercal} + f\Grad{\bm g}, \\
\nabla({\bm f} \cdot {\bm g}) & = \Grad{\bm f}^{\intercal }{\bm g} + \Grad{\bm g}^{\intercal}{\bm f},
\end{align*}
and the identity $\bn \cdot \nab \psi^e = \bn \cdot \nab_\Gamma \psi=0$ 
on $\Gamma$
to conclude
\begin{equation}
\label{eqn:Intermed}
\begin{split}
\Grad (\nabla_{\Gamma}\psi)^e & = \Grad({\bf P}\nabla\psi^e) \\
& = \Grad(\nabla\psi^e - (\bn \cdot \nabla\psi^e)\bn) \\
& = \nabla^{2}\psi^e - \bn\nabla(\bn \cdot \nabla\psi^e)^{\intercal}
- (\bn \cdot \nabla\psi^e)\Grad\bn \\
& = \nabla^{2}\psi^e - \bn \otimes \left(\Grad\bn^{\intercal}\nabla\psi^e + \nabla^{2}\psi^{e}\bn\right)\\ 
& = {\bf P} \nabla^{2}\psi^e - \bn \otimes ({\bf H}\nabla\psi^e).
\end{split}
\end{equation}
Since ${\bf P}^{2} = {\bf P}$ and ${\bf P}\bn = 0$, 
there holds
\[
\nabla_{\Gamma}^{2}\psi = {\bf P}\Grad(\nabla_{\Gamma}\psi)^e {\bf P} = {\bf P}\nabla^{2}\psi^e {\bf P}.
\]
\end{proof}

\begin{lemma}
For a vector $\bxi = \left[\xi_{1}, \xi_{2}, \xi_{3}\right]^\intercal \in \mathbb{R}^{3}$, define
the $3\times 3$ skew-symmetric matrix
\[
\bxi^{\times} = {\rm mskw}(\bxi) = 
\left[
\begin{array}{c c c}
0 & - \xi_{3} & \xi_{2} \\
\xi_{3} & 0 & - \xi_{1} \\
- \xi_{2} & \xi_{1} & 0
\end{array}
\right].
\]
Then there holds
\begin{align}\label{eqn:streamGrad}
\Grad_{\Gamma}\bcurl_\Gamma \psi & = \bn^\times \nab_\Gamma^2 \psi = \bn^{\times}\underline{D}^{2}\psi,\\
\label{eqn:deformGrad}
H_\Gamma(\psi) & = \frac12 (\bn^\times \nab_\Gamma^2 \psi - \nab_\Gamma^2 \psi \bn^\times) = \frac{1}{2}\left(\bn^{\times}\underline{D}^{2}\psi - \underline{D}^{2}\psi^\intercal\bn^{\times}\right).
\end{align}
\end{lemma}
\begin{proof}
Note that
\[
\bn^{\times}{\bf P} = \bn^{\times} = {\bf P}\bn^{\times},
\quad\text{and}\quad {\bf H}{\bf P} = {\bf H} = {\bf P}{\bf H}.
\]
Thus, using the identity
$\bn\times \nab_\Gamma \psi = \bn^\times \nab_\Gamma \psi = - (\nab_\Gamma \psi)^\times \bn$, we have
\begin{align*}
\Grad_{\Gamma}\mathbf{curl}_{\Gamma}\psi
& = {\bf P}\Grad\left(\bn \times \nabla_{\Gamma}\psi\right)^e {\bf P} \\
& = {\bf P}\left(\bn^{\times}\Grad(\nabla_{\Gamma}\psi)^e - (\nabla_{\Gamma}\psi)^{\times}{\bf H}\right){\bf P} \\
& = \bn^{\times}\left(\Grad(\nabla_{\Gamma}\psi)^e\right){\bf P} - {\bf P}(\nabla_{\Gamma}\psi)^{\times}{\bf H}.
\end{align*}
Next, a direct calculation shows 

\begin{align*}
{\bf P}({\bf P}\bxi)^{\times} & =
\begin{bmatrix}
n_{1}\left(n_{3}\xi_{2} - n_{2}\xi_{3}\right) & n_{2}(n_{3}\xi_{2} - \xi_{3}n_{2}) & n_3(\xi_{2}n_{3} - n_{2}\xi_{3}) \\
n_1(\xi_{3}n_{1} - n_{3}\xi_{1}) & n_{2}\left(n_{1}\xi_{3} - n_{3}\xi_{1}\right) & n_3(n_{1}\xi_{3} - \xi_{1}n_{3}) \\
n_{1}(n_{2}\xi_{1} - \xi_{2}n_{1}) & n_2(\xi_{1}n_{2} - n_{1}\xi_{2}) & n_{3}\left(n_{2}\xi_{1} - n_{1}\xi_{2}\right)
\end{bmatrix}
\\
& = (\bxi \times \bn)\otimes \bn\qquad \forall \bxi\in \bbR^3.
\end{align*}
Hence, using ${\bf H}\bn = 0$ we obtain 
\[
{\bf P}(\nabla_{\Gamma}\psi)^{\times}{\bf H} = {\bf P}\left({\bf P}\nabla\psi^e\right)^{\times}{\bf H} = (\nabla\psi^e \times \bn) \otimes \bn {\bf H} = 0,
\]
and therefore,
\[
\Grad_{\Gamma}\mathbf{curl}_{\Gamma}\psi
= \bn^{\times}\left(\Grad(\nabla_{\Gamma}\psi)^e\right){\bf P}.
\]
We then use \eqref{eqn:Intermed} and \eqref{nHessIdentity}, along 
with $\bn^{\times}\bn = 0$ to obtain
\begin{align*}
\Grad_{\Gamma}\mathbf{curl}_{\Gamma}\psi& = \bn^{\times}\left({\bf P}\nabla^{2}\psi^e - \bn\otimes ({\bf H}\nab \psi^e) \right) {\bf P} \\
& = \bn^{\times}{\bf P}\nabla^{2}\psi^e {\bf P}
 = \bn^{\times}\nabla_{\Gamma}^{2}\psi.
\end{align*}
The other equalities in \eqref{eqn:streamGrad}--\eqref{eqn:deformGrad} then
follow from the identities $\left(\bxi^{\times}\right)^{\intercal} = - \bxi^{\times}$, $\nab_\Gamma^2 \psi = {\bf P} \underline{D}^2 \psi$, and $\bn^\times \bn = 0$.
\end{proof}

\subsection{Integration-by-parts identities for the surface biharmonic operator}
In this section, we derive integration-by-parts identities
for the biharmonic operator $\Delta_\Gamma^2(\cdot):=\Delta_\Gamma (\Delta_\Gamma(\cdot))$
over a a simply connected sub-domain $S\subset \Gamma$ with boundary.
Such results will motivate the design
of the $C^0$ interior penalty method discussed in the next section.
Our starting point is a well-known integration-by-parts identity.
\begin{lemma}\label{lem:IBP}
Let $S\subset \Gamma$, and let $\bmu_S$ denote the outward unit co-normal of $\p S$.
Then there holds for all sufficiently 
smooth functions $\phi,\psi$ on $S$,
  \begin{equation}\label{eqn:IBP}
\int_S \psi \underline{D}_i \phi = -\int_{S} \phi \underline{D}_i \psi +\int_S \phi \psi {\rm tr}({\bf H}) n_i +\int_{\p S} \phi \psi (\mu_S)_i\quad i=1,2,3.
\end{equation}
\end{lemma}

Repeated applications of Lemma \ref{lem:IBP} immediately yield the following result.
\begin{lemma}\label{lem:LapLap}
There holds for sufficient smooth functions $\phi$ and $\psi$,
\begin{equation}
\label{eqn:LapLap}
\begin{split}
    \int_S (\Delta_\Gamma^2 \phi) \psi 
    & = \int_S (\Delta_\Gamma \phi)(\Delta_\Gamma \psi) - \int_{\p S} \Delta_\Gamma \phi (\nab_{\Gamma} \psi\cdot \bmu_S) 
    +\int_{\p S} (\nab_\Gamma \Delta_\Gamma \phi \cdot \bmu_S) \psi.
    \end{split}
\end{equation}
\end{lemma}

In addition, applications of Lemma \ref{lem:IBP}
also lead to an integration-by-parts identity involving
the projected surface Hessian operator $\nab_\Gamma^2(\cdot)$.
The derivation of the following lemma is a bit more involved, and its proof
is given in Appendix \ref{app:ProofHessForm}.
\begin{lemma}\label{lem:HessForm}
There holds for sufficiently smooth functions $\phi,\psi$,
\begin{equation}\label{eqn:HessForm}
\begin{split}
\int_S (\Delta_\Gamma^2 \phi) \psi 
&  = \int_S \left(\nab^2_\Gamma \phi:\nab^2_\Gamma \psi +K \nab_\Gamma \phi\cdot \nab_\Gamma \psi\right)\\ 
&\qquad-\int_{\p S} \bmu_S^\intercal \nab_\Gamma^2 \phi \nab_\Gamma \psi
  +\int_{\p S} (\nab_\Gamma \Delta_\Gamma \phi\cdot \bmu_S)\psi,
\end{split}
\end{equation}
where $K$ is the Gauss curvature of $\Gamma$.
\end{lemma}

Finally, we derive an integration-by-parts identity
for the biharmonic operator such that, 
when applied to the surface stream function problem \eqref{eqn:streamfunction},
the integrands over $S$ do not explicitly depend on the Gauss curvature of $\Gamma$.
\begin{lemma}\label{lem:HessForm2}
There holds for sufficiently smooth functions $\phi$ and $\psi$,
\begin{equation}\label{eqn:ThirdBiharmonicIdentity}
\begin{split}
\int_S (\Delta_\Gamma^2& \phi)\psi
 = 2\int_S \left(H_\Gamma(\phi):H_\Gamma(\psi) + K\nab_\Gamma \phi\cdot \nab_\Gamma \psi\right)\\
& -
2\int_{\p S} (\bt_S^\intercal H_\Gamma(\phi) \bmu_S) (\nab_\Gamma \psi\cdot \bmu_S)-2 \int_{\p S} (\bmu_S^\intercal \nab_\Gamma^2 \phi \bt_S) (\bt_S\cdot \nab_\Gamma \psi)
   +\int_{\p S} (\nab_\Gamma \Delta_\Gamma \phi\cdot \bmu_S)\psi,
\end{split}
\end{equation}
where $\bt_S = \bn_S\times \bmu_S$ is the tangent vector of $\p S$.
Consequently, if $\phi\in H^4(\Gamma)$ satisfies \eqref{eqn:streamfunction},
then there holds for all smooth $\psi$,
\begin{equation}
\label{eqn:StreamForm2}
\begin{split}
 &\int_S \left(H_\Gamma(\phi) : H_\Gamma(\psi)+\nab_\Gamma \phi \cdot \nab_\Gamma \psi\right) 
 -  \int_{\p S}  (\bt_S^\intercal H_\Gamma(\phi) \bmu_S)(\nab_\Gamma \psi\cdot \bmu_S)\\ 
 &\ \ + \frac12 \int_{\p S} \left((2K-2) \nab_\Gamma \phi + \nab_\Gamma \Delta_\Gamma \phi) \cdot \bmu_S\right) \psi
-\int_{\p S}
    (\bmu^T_S\nab_\Gamma^2 \phi\bt_S)(\bt_S\cdot \nab_\Gamma \psi) = -\int_S {\rm curl}_\Gamma {\bm f} \psi.
\end{split}
\end{equation}
\end{lemma}
\begin{proof}
We begin with the following pointwise identity,
which holds for all sufficiently smooth tangential vector fields
$\bu$ \cite{JankuhnEtal18}:
\[
{\bf P} {\rm div}_{\Gamma}(\Grad_{\Gamma} \bu^\intercal ) = \nab_{\Gamma}({\rm div}_\Gamma \bu)+K \bu.
\]
Setting $\bu = \bcurl_{\Gamma} \phi$,
so that ${\rm div}_{\Gamma} \bu = 0$
and $\Grad_{\Gamma} \bu^\intercal = -\nab_\Gamma^2 \phi \bn^\times$ (cf.~\eqref{eqn:streamGrad}) yields
\[
-P {\rm div}_{\Gamma}(\nab_\Gamma^2 \phi \bn^\times) = K \bcurl_{\Gamma} \phi.
\]
We then take the dot product of this expression 
with $\bcurl_\Gamma \psi$, integrate over $S$, and integrate by parts to obtain
\begin{align*}
    \int_S K \nab_\Gamma \phi\cdot \nab_\Gamma \psi
    & =  \int_S K \bcurl_{\Gamma} \phi\cdot \bcurl_\Gamma \psi\\
    & = \int_S (\nab_\Gamma^2 \phi\bn^\times):(\Grad_\Gamma \bcurl_\Gamma \psi) -\int_{\p S}
    (\nab_\Gamma^2 \phi \bn^\times \bmu_S) \cdot \bcurl_\Gamma \psi\\
& =     \int_S (\nab_\Gamma^2 \phi\bn^\times):(\bn^\times \nab^2_\Gamma \psi) -\int_{\p S}
    (\nab_\Gamma^2 \phi\bt_S) \cdot \bcurl_\Gamma \psi.
\end{align*}
%
We  then write, on $\p S$,
\begin{align*}
\bcurl_\Gamma \psi = \bn\times \nab_{\Gamma} \psi
%
& = (\bmu_S\cdot \nab_\Gamma \psi) \bt_S - (\bt_S \cdot \nab_\Gamma \psi)\bmu_S,
\end{align*}
to arrive at
\begin{equation}
\label{eqn:GaussGone}
\begin{split}
    \int_S K \nab_\Gamma \phi\cdot \nab_\Gamma \psi
& =     \int_S (\nab_\Gamma^2 \phi\bn^\times):(\bn^\times \nab^2_\Gamma \psi)\\
&\qquad -\int_{\p S}
    (\bt^\intercal_S\nab_\Gamma^2 \phi\bt_S)(\bmu_S\cdot \nab_\Gamma \psi)
    +\int_{\p S}
    (\bmu_S^\intercal \nab_\Gamma^2 \phi\bt_S)(\bt_S\cdot \nab_\Gamma \psi).
\end{split}
\end{equation}
Inserting this identity into 
\eqref{eqn:HessForm} and applying the
algebraic identities 
\begin{align*}
\frac12 \left(\nab^2_\Gamma \phi:\nab^2_\Gamma \psi -
(\nab_\Gamma^2 \phi \bn^\times):(\bn^\times \nab_\Gamma^2 \psi)\right) &= H_\Gamma(\phi):H_\Gamma(\psi),\\
\frac12 \left(\bmu^\intercal_S \nab_\Gamma^2 \phi \bmu_S - \bt^\intercal_S \nab_\Gamma^2 \phi \bt_S\right) &= \bt_S^\intercal H_{\Gamma}(\phi)\bmu_S,
\end{align*}
then yield
\begin{equation*}
\begin{split}
\int_S (\Delta_\Gamma^2 \phi) \psi 
&  = \int_S \left(\nab^2_\Gamma \phi:\nab^2_\Gamma \psi +
(\nab_\Gamma^2 \phi \bn^\times):(\bn^\times \nab_\Gamma^2 \psi)\right)\\
&\qquad-\int_{\p S} (\bmu^\intercal_S \nab_\Gamma^2 \phi \bmu_S+\bt_S^\intercal \nab_\Gamma^2 \phi \bt_S) (\nab_\Gamma \psi\cdot \bmu_S)
   +\int_{\p S} (\nab_\Gamma \Delta_\Gamma \phi\cdot \bmu_S)\psi\\
& =2 \int_S \left( H_\Gamma(\phi):H_\Gamma(\psi) + (\nab_\Gamma^2 \phi \bn^\times):(\bn^\times \nab_\Gamma^2 \psi)\right)\\
&\qquad-\int_{\p S} (\bmu^\intercal_S \nab_\Gamma^2 \phi \bmu_S+\bt_S^\intercal \nab_\Gamma^2 \phi \bt_S) (\nab_\Gamma \psi\cdot \bmu_S)
   +\int_{\p S} (\nab_\Gamma \Delta_\Gamma \phi\cdot \bmu_S)\psi\\
%
%
& = 
2\int_S \left(H_\Gamma(\phi):H_\Gamma(\psi) + K\nab_\Gamma \phi\cdot \nab_\Gamma \psi\right)\\
&\qquad-
2\int_{\p S} (\bt^\intercal_S H_\Gamma(\phi) \bmu_S) (\nab_\Gamma \psi\cdot \bmu_S)-2 \int_{\p S} (\bmu^\intercal_S \nab_\Gamma^2 \phi \bt_S) (\bt_S\cdot \nab_\Gamma \psi)
   +\int_{\p S} (\nab_\Gamma \Delta_\Gamma \phi\cdot \bmu_S)\psi.
 \end{split}
 \end{equation*}
Thus, \eqref{eqn:ThirdBiharmonicIdentity} holds. 
The identity \eqref{eqn:StreamForm2} follows from \eqref{eqn:ThirdBiharmonicIdentity}
and a simple application of the divergence theorem.

\end{proof}

\section{Approximate Geometries and Meshes}\label{sec-Mesh}
Let $\bar \Gamma_h$ be a polyhedral approximation of $\Gamma$ with triangular
faces, and assume  $\bar \Gamma_h\subset U_\delta$ and 
$d(x) = O(h^2)$ for all $x\in \bar \Gamma_h$. 
The set of faces of $\bar \Gamma_h$
is denoted by $\bar \calT_h$, which we assume to be shape-regular. 
For simplicity in presentation,
we assume $\bar \calT_h$ is quasi-uniform and 
that the vertices of each $\bar T\in \bar \calT_h$ lie on $\Gamma$, i.e., 
$\bar \Gamma_h$ is the continuous, piecewise linear interpolant of $\Gamma$.
We  set $h_{\bar T} = {\rm diam}(\bar T)$ for all $\bar T\in \bar \calT_h$,
denote by $\bar \calE_h$ the set of edges in $\bar \calT_h$,
and set $h_{\bar e} = {\rm diam}(\bar e)$. Note the quasi-uniform
assumption implies $h_{\bar e} \approx h_{\bar T} \approx h$ for all $\bar T\in \bar \calT_h$
and $\bar e\in \bar \calE_h$.

For $k_g\in \mathbb{N}$ and $\bar T\in \bar \calT_h$, let $\{\bar{x}_j\}_{j=1}^{N_{k_g}} \subset {\rm cl}(\bar T)$ (with $N_{k_g} = \binom{k_g+2}{2}$)
be the standard Lagrange nodal points of $\bar T$, and 
let $\{\bar \phi_i\}_{i=1}^{N_{k_g}} \subset \mathbb{P}_{k_g}(\bar T)$ be 
the associated nodal basis functions, i.e., $\bar \phi_i(\bar x_j) = \delta_{i,j}$.
We then define the approximate closest-point projection
\[
\bp_{k_g} |_{\bar T}(\bar x) = \sum_{i=1}^{N_{k_g}} \bp(\bar x_i) \bar \phi_i(\bar x)\qquad \bar x\in {\rm cl}(\bar T),
\]
so that $\bp_{k_g}|_{\bar T}(\bar x_i) = \bp(\bar x_i)$ for $i=1,\ldots,N_{k_g}$, 
i.e., $\bp_{k_g}|_{\bar T}$ is the $k_g$-degree Lagrange interpolant of $\Gamma$.
The high-order mesh and associated surface are then defined as
\[
\calT_{h,k_g} = \{\bp_{k_g}(\bar T)\ \forall \bar T\in \bar \calT_h\},\qquad \Gamma_{h,k_g} = {\rm int}\left(\bigcup_{T\in \calT_{h,k_g}} {\rm cl}(T)\right),
\]
whose set of edges are given by
\[
\calE_{h,k_g} = \{\bp_{k_g}(\bar e)\ \forall \bar e\in \bar \calE_h\}.
\]
We set $h_T = h_{\bar T}$ and $h_e = h_{\bar e}$, where $T = \bp_{k_g}(\bar T)$ and $e = \bp_{k_g}(\bar e)$.

To simplify the presentation, we will drop the subscript $k_g$ and simply write $\calT_h$, $\Gamma_h$, and $\calE_h$
for  $\calT_{h,k_g}$, $\Gamma_{h,k_g}$, and $\calE_{h,k_g}$, respectively. 
We let $\bn_T$ denote the outward unit normal
of $T\in \calT_h$, let $\bmu_T$ be the outward unit
co-normal of $\p T$, and set $\bt_T = \bn_T\times \bmu_T$. We also let $\bn_h$ be the outward unit
normal of $\Gamma_h$, so that $\bn_h|_T = \bn_T$ for all $T\in \calT_h$.  Likewise, we
let $\bmu_h$ and $\bt_h$ be defined such that $\bmu_h|_{\p T} = \bmu_T$ and $\bt_h|_{\p T} = \bt_T$ for all $T\in \calT_h$.
By properties
of the Lagrange interpolant, we have (cf.~\cite{Demlow09})
\begin{equation}
    \label{eqn:GeoBounds}
    \|d\|_{L^\infty(\Gamma_h)}\lesssim h^{k_g+1},\qquad \|\bn^e-\bn_h\|_{L^\infty(\Gamma_h)}\lesssim h^{k_g},
\end{equation}
where we use the notation $A\lesssim B$ 
to mean $A\le c B$  for some constant
$c>0$ independent of $h.$ We also use $A \approx B$ to mean $A\lesssim B$ and $B\lesssim A$.

We also map these mesh objects onto the exact surface $\Gamma$ via the closest-point projection.
Define
\begin{alignat*}{2}
%
\calT_h^\ell &= \{T^\ell:=\bp(T):\ T\in \calT_h\},\qquad &&\calE_h^\ell = \{e^\ell:=\bp(e):\ e\in \calE_h\}.
\end{alignat*}
We let $\bn_{T^\ell}$ denote
the outward unit normal of $T^\ell\in \calT_h^\ell$ (so that $\bn|_{T^\ell} = \bn_{T^\ell}$),
 let $\bmu_{T^\ell}$ denote the outward unit co-normal of $\p T^\ell$,
 and set $\bt_{T^\ell} = \bn_{T^\ell}\times \bmu_{T^\ell}$. 
 Let $\bmu$ and $\bt$ be defined such that $\bmu|_{\p T^\ell} = \bmu_{T^\ell}$ and $\bt|_{\p T^\ell} = \bt_{T^\ell}$. 
In the rest of the paper, we will drop the superscript $e$ and simply
write $\bn_{T^\ell}$, $\bn$,  $\bmu_{T^\ell}$, etc. for their respective extensions.

We use the notation 
\begin{align*}
(\psi,\chi)_{\mathcal{S}_h} = \sum_{S\in \mathcal{S}_h} \int_S \psi \circ \chi,\quad \text{and}\quad \|\psi\|_{L^2(\mathcal{S}_h)} = \sqrt{(\psi,\psi)_{\mathcal{S}_h}},
\end{align*}
where $\mathcal{S}_h$ is either a set of faces or edges,
and $\circ$ is either a product, dot product, or Frobenius product depending on whether $\psi$ and $\chi$ are scalar,
vector, or matrix-valued functions.

For $e\in \calE_h$ with $e=\p T_+\cap \p T_-$ ($T_{\pm}\in \calT_h$), we define the average of a piecewise smooth scalar, vector, or matrix-valued
function $w$ across $e$ as
\[
\avg{w}|_e = \frac12 (w_++w_-),
\]
where $w_{\pm} = w|_{T_{\pm}}$.
For a piecewise smooth vector-valued function $\bv$, we define its jump across
$e$ as
\[
\jump{\bv}|_e = \bv_+\cdot \bmu_+ + \bv_- \cdot \bmu_-, 
\]
where $\bmu_{\pm} = \bmu_{T_{\pm}}$ is the outward unit co-normal of $\p T_{\pm}$ restricted to $e$.
Analogous definitions of averages and jumps are extended to the mesh $\calT_h^\ell$.

\subsection{Mappings between $\Gamma$, $\Gamma_h$, and $\bar \Gamma_h$}

For a scalar function $\psi$ defined on the exact surface $\Gamma$,
we recall that its extension $\psi^e:U_\delta\to \bbR$ is given by
$\psi^e= \psi \circ \bp$.
For a scalar function $\psi$ defined on 
the discrete surface $\Gamma_h$, 
we let $\tilde \psi:\Gamma \to \bbR$ 
be defined by
\[
\tilde \psi = \psi\circ x_h,
\]
where $x_h(x)$ (with $x\in \Gamma$) satisfies $\bp(x_h) = x$, i.e., $x_h = \bp|_{\Gamma_h}^{-1}$.
We then set the lift of $\psi:\Gamma_h\to \bbR$ 
as
\[
\psi^\ell = \tilde \psi \circ \bp\qquad \text{on }U_\delta.
\]

Next for $x\in \Gamma_h$, let $\mu_h(x)$ satisfy $\mu_h(x)d\sigma_h(x) = d\sigma(\bp(x))$,
where $d\sigma$ and $d\sigma_h$ are the surface measures on $\Gamma$ and $\Gamma_h$ respectively.
In particular, we have
\begin{align*}
\int_{\Gamma_h} \psi^e \mu_h 
= \int_\Gamma \psi\qquad \forall \psi\in L^1(\Gamma),
\end{align*}
and so
\begin{align*}
\int_{\Gamma_h} \psi = \int_\Gamma (\mu_h^{-1} \psi)^\ell
\qquad \forall \psi\in L^1(\Gamma_h).
\end{align*}
Likewise, for $x\in e\in \calE_h$, let $\mu_e$ satisfy
$\mu_e(x) ds_h(x) = ds(\bp(x))$, where $ds$ and $ds_h$ 
are the measures on $e^\ell = \bp(e)$ and $e$, respectively.
These functions satisfy \cite{Demlow09,CockburnDemlow16}
\begin{align*}
\mu_h(x) &= \bn(x)\cdot \bn_h(x) \prod_{i=1}^2 (1-d(x)\kappa_i(x))\qquad x\in \Gamma_h,\\
\mu_e&= |\Grad \bp \bt_e| = \big| ({\bf P}-d {\bf H})\bt_e\big|,
\end{align*}
where $\{\kappa_i\}$ are the principal curvatures and 
$\bt_e = \bt_h|_e$.
Recalling \eqref{eqn:GeoBounds}, we have
\begin{equation}
1-\bn\cdot \bn_h = \frac12(\bn\cdot \bn - 2\bn\cdot \bn_h+ \bn_h\cdot \bn_h) = \frac12 |\bn-\bn_h|^2 = \mathcal{O}(h^{2k_g}), \label{eqn:one-nnh}
\end{equation}
and so 
\begin{equation}\label{eqn:muCloseToOne}
|1-\mu_h|= \mathcal{O}(h^{k_g+1}).
\end{equation}
Likewise we have
\begin{align}\label{eqn:muECloseToOne}
    |1-\mu_e|= \mathcal{O}(h^{k_g+1}).
\end{align}
From these estimates and the chain rule (cf.~Lemma \ref{lem:ChainRuleFun})
we have for all $\psi\in H^1(T)$ (cf.~\cite[(2.15)--(2.17)]{Demlow09})
\begin{subequations}
\label{eqn:NormEquivy}
\begin{equation}
\begin{aligned}
\|\psi\|_{L^2(T)}&\approx \|\psi^\ell \|_{L_2(T^\ell)},\qquad
&& \|\nab_{\Gamma_h} \psi\|_{L^2(T)}\approx \|\nab_\Gamma \psi^\ell \|_{L_2(T^\ell)},\\
\|\psi\|_{L^2(\p T)}&\approx \|\psi^\ell \|_{L_2(\p T^\ell)}.
\end{aligned}
\end{equation}
Likewise, for $\psi\in H^m(T)$ with $m\ge 1$, there holds
\begin{align}
|\psi|_{H^m(T)}\lesssim \sum_{j=1}^m |\psi^\ell|_{H^j(T^\ell)},\quad \text{and}\quad
|\psi^\ell|_{H^m(T^\ell)}\lesssim \sum_{j=1}^m |\psi|_{H^j(T)}.
\end{align}
\end{subequations}
By properties of the Lagrange interpolant,
we also have (cf.~\cite[(2.18)--(2.20)]{Demlow09})
\begin{equation}
\label{eqn:NormEquivyLag}
\begin{aligned}
\|\bar \psi\|_{L^2(\bar T)}&\approx \|\psi\|_{L^2(T)},\qquad && \|\nab_{\bar \Gamma_h} \bar \psi\|_{L^2(\bar T)}\approx \|\nab_{\Gamma_h} \psi \|_{L_2(T)},\\
\|\bar \psi\|_{L^2(\p \bar T)}&\approx \|\psi \|_{L_2(\p T)},
\end{aligned}
\end{equation}
where $T\in \calT_h$ and $\bar T\in \bar \calT_h$ satisfy $T = \bp_{k_g}(\bar T)$ and $\bar \psi = \psi \circ \bp_{k_g}$.


\begin{lemma}\label{lem:GeomEstimates1}
The following estimates hold on each $T\in \calT_h$:
\begin{align}
\left\|{\bf P}_{h}\bn\right\|_{L^{\infty}\left(T\right)} & \lesssim h^{k_{g}}, \label{GeoApprox.1} \\
\left\|{\bf P}\bn_{h}\right\|_{L^{\infty}\left(T\right)} & \lesssim h^{k_{g}}, \label{GeoApprox.2} \\
\left\|\bmu - {\bf P}\bmu_h\right\|_{L^{\infty}\left(\p T\right)} & \lesssim h^{k_{g} + 1}, \label{GeoApprox.3} \\
\left\|\bmu - \bmu_h\right\|_{L^{\infty}(\p T)} & \lesssim h^{k_g},  \label{GeoApprox.4} \\
\left\|\bt - \bt_h\right\|_{L^{\infty}(\p T)} & \lesssim h^{k_g}. \label{GeoApprox.5}
\end{align}
\end{lemma}
\begin{proof}
The estimates \eqref{GeoApprox.1}--\eqref{GeoApprox.2} and 
\eqref{GeoApprox.4}--\eqref{GeoApprox.5} follow from \eqref{eqn:GeoBounds}.
The proof of \eqref{GeoApprox.3} is found in \cite{OlshanskiiReusken14}.

\end{proof}

\section{The $C^0$ Interior Penalty Method}\label{sec-Method}

For an integer $k\ge 2$, we define  families of finite element spaces 
defined on the affine mesh, the polynomial-mapped mesh,
and the surface mesh:
\begin{align*} 
\bar V_h &= \{\bar \psi\in C^0(\bar \Gamma_h):\ \bar \psi|_{\bar T}\in \mathbb{P}_k(\bar T)\ \forall \bar T\in \bar\calT_h\},\qquad
 V_h = \{\psi\in C^0(\Gamma_h):\  \psi\circ \bp_{k_g} \in \bar V_h\},\\
 V_h^\ell &= \{\psi\in C^0(\Gamma):\ \psi \circ \bp \in V_h\},
\end{align*}
where $\mathbb{P}_k(\bar T)$ denotes the space of polynomials of degree $\le k$ on $\bar T$.
We further set $V_{h,0} = V_h\cap L^2_0(\Gamma_h)$,
where $L^2_0(\Gamma_h)$ is the space of square integrable 
functions with vanishing mean.
Likewise, we define the spaces of piecewise smooth functions
\begin{align*} 
\bar W &= \{\bar w\in C^0(\bar \Gamma_h):\ \bar w|_{\bar T}\in H^3(\bar T)\ \forall \bar T\in \bar \calT_h\},\qquad
 W = \{w\in C^0(\Gamma_h):\ w|_{T}\in H^3(T)\ \forall T\in \calT_h\},\\
 W^\ell & = \{w\in C^0(\Gamma):\ w|_{T^\ell} \in H^3(T^\ell)\ \forall T^\ell\in \calT_h^\ell\},
\end{align*}
and note the obvious inclusions $\bar V_h\subset \bar W$, $V_h\subset W$, and $V_h^\ell \subset W^\ell$.
Similar to above, we set $W_0 = W\cap L^2_0(\Gamma_h)$.

To derive the $C^0$ IP method, we assume for the moment
that $\phi$ is a smooth solution to \eqref{eqn:streamfunction}. 
We use the identity \eqref{eqn:StreamForm2} with $S = T^\ell \in \calT_h^\ell$
and sum over elements to obtain, for all $\psi \in V_h^\ell$,
\begin{align*}
- \int_\Gamma ({\rm curl}_{\Gamma} {\bm f}) \psi
= \sum_{T^\ell \in \calT^\ell_h} \left(\int_{T^\ell} H_{\Gamma}(\phi) : H_{\Gamma}(\psi) +\int_{T^\ell} \nab_{\Gamma} \phi \cdot \nab_{\Gamma} \psi - \int_{\p T^\ell} (\bt_{T^\ell}^\intercal H_\Gamma(\phi) \bmu_{T^\ell} )(\nab_\Gamma \psi \cdot \bmu_{T^\ell})\right),
\end{align*} 
where we used the continuity of $\psi$.
Using the temporary smoothness assumption of $\phi$, we have
\[
\sum_{T^\ell \in \calT_h^\ell} \int_{\p T^\ell} (\bt_{T^\ell}^\intercal H_\Gamma(\phi) \bmu_{T^\ell} )(\nab_\Gamma \psi \cdot \bmu_{T^\ell}) = \sum_{e^\ell \in \calE_h^\ell} \int_{e^\ell} \avg{\bt^\intercal H_{\Gamma}(\phi) \bmu} \jump{\nab_{\Gamma} \psi}.
\]
We then add standard symmetry and stabilization terms to obtain the identity
\begin{align*}
\ell(\psi):=- \int_\Gamma ({\rm curl}_{\Gamma} {\bm f}) \psi
& = 
\sum_{T^\ell \in \calT^\ell_h} \int_{T^\ell} H_{\Gamma}(\phi) : H_{\Gamma}(\psi) +\int_{\Gamma} \nab_{\Gamma} \phi \cdot \nab_\Gamma \psi
-\sum_{e^\ell \in \calE_h^\ell} \int_{e^\ell} \avg{\bt^\intercal H_{\Gamma}(\phi) \bmu} \jump{\nab_{\Gamma} \psi}\\
&\qquad -\sum_{e^\ell \in \calE_h^\ell} \int_{e^\ell} \avg{\bt^\intercal H_{\Gamma}(\psi) \bmu} \jump{\nab_{\Gamma} \phi}
+\sigma \sum_{e^\ell \in \calE_h^\ell} {h^{-1}_{e^\ell}} \int_{e^\ell} \jump{\nab_{\Gamma} \phi} \jump{\nab_{\Gamma} \psi}=:a_h^\ell(\phi,\psi),
\end{align*}
where $\sigma>0$ is a penalty parameter.

This calculation motivates the method: Find $\phi_h\in V_{h,0}$
satisfying
\begin{align} \label{eqn:C0IPMethod}
a_h(\phi_h,\psi) =  \ell_h(\psi):=\int_{\Gamma_h}  {\bm f}_h \cdot \bcurl_{\Gamma_h} \psi\qquad \forall \psi\in V_{h,0},
\end{align}
where ${\bm f}_h$ is an approximation of ${\bm f}$, defined on $\Gamma_h$, and 
\begin{align*}
a_h(\phi_h,\psi)
:&=
\sum_{T \in \calT_h} \int_{T} H_{\Gamma_h}(\phi_h):H_{\Gamma_h}(\psi) +\int_{\Gamma_h} \nab_{\Gamma_h} \phi_h \cdot \nab_{\Gamma_h} \psi
-\sum_{e \in \calE_h} \int_{e} \avg{\bt^\intercal_h H_{\Gamma_h}(\phi_h) \bmu_h} \jump{\nab_{\Gamma_h} \psi}\\
&\qquad -\sum_{e \in \calE_h} \int_{e} \avg{\bt^\intercal_h H_{\Gamma_h}(\psi) \bmu_h} \jump{\nab_{\Gamma_h} \phi_h}
+\sigma \sum_{e \in \calE_h} {h_{e}^{-1}} \int_{e} \jump{\nab_{\Gamma_h} \phi_h} \jump{\nab_{\Gamma_h} \psi}.
\end{align*}
Without loss of generality, we assume ${\bm f}_h\cdot \bn_h = 0$.

\subsection{Stability and Continuity estimates}
To start the analysis of the $C^0$ IP method \eqref{eqn:C0IPMethod}
we define the following three norms on $W_0$:
\begin{align*}
 \|\psi\|_{2,h}^2 &= \sum_{T\in \calT_h} \|H_{\Gamma_h}(\psi)\|_{L^2(T)}^2 + \|\nab_{\Gamma_h} \psi\|_{L^2(\Gamma_h)}^2+\sum_{e\in \calE_h} h_e^{-1} \|[\nab_{\Gamma_h} \psi]\|_{L^2(e)}^2,\\
 \tbar{\psi}^2_{2,h}
 & = \|\psi\|_{2,h}^2 + \sum_{e\in \calE_h} h_e \|\avg{H_{\Gamma_h}(\psi)}\|_{L^2(e)}^2,\\
 \tbar{\psi}^2_{H_h} 
 & = \tbar{\psi}^2_{2,h}+ \sum_{T\in \calT_h} |\psi|_{H^2(T)}^2 + \sum_{T\in \calT_h} h_T^2 |\psi|_{H^3(T)}^2.
\end{align*}
Analogous norms are also defined on the exact surface $\Gamma$:
\begin{align*}
 \|\psi\|_{2}^2 &= \sum_{T^\ell\in \calT^\ell_h} \|H_{\Gamma}(\psi)\|_{L^2(T^\ell)}^2 + \|\nab_{\Gamma} \psi\|_{L^2(\Gamma)}^2+\sum_{e^\ell \in \calE_h^\ell} h_{e^\ell}^{-1} \|[\nab_{\Gamma} \psi]\|_{L^2(e^\ell)}^2,\\
 \tbar{\psi}^2_{2}
 & = \|\psi\|_{2}^2 + \sum_{e^\ell\in \calE^\ell_h} h_{e^\ell} \|\avg{H_{\Gamma}(\psi)}\|_{L^2(e^\ell)}^2,\\
 \tbar{\psi}^2_{H} 
 & = \tbar{\psi}^2_{2}+ \sum_{T^\ell\in \calT^\ell_h} |\psi|_{H^2(T^\ell)}^2 +\sum_{T^\ell\in \calT^\ell_h} h_{T^\ell}^2 |\psi|_{H^3(T^\ell)}^2.
\end{align*}

The following proposition states some standard trace
and inverse estimates.
\begin{proposition}\label{prop:Trace}
Let $T\in \calT_h$ and $T^\ell\in \calT_h^\ell$.
Then there holds
\begin{equation}
\label{tracineq}
\begin{aligned}
\|\psi\|_{L^2(\p T)}^2
&\lesssim h_T^{-1} \|\psi\|^2_{L^2(T)}+h_T \|\nab_{\Gamma_h} \psi\|_{L^2(T)}^2\qquad &&\forall \psi\in H^1(T),\\
\|\psi\|_{L^2(\p T^\ell)}^2
&\lesssim h_T^{-1} \|\psi\|^2_{L^2(T^\ell)}+h_T \|\nab_{\Gamma} \psi\|_{L^2(T^\ell)}^2\qquad &&\forall \psi\in H^1(T^\ell).
\end{aligned}
\end{equation}
Moreover, for any $\psi\circ \bp_{k_g} \in \mathbb{P}_k(\bar T)$ with 
$\bar T\in \bar \calT_h$ and $k\in \mathbb{N}$, there holds (with $T = \bp_{k_g}(\bar T))$
the inverse estimates ($m\in \mathbb{N}$)
\begin{equation}
\label{inveineq}
\begin{split}
\|\psi\|_{H^m(T)}&\lesssim h_T^{-m} \|\psi\|_{L^2(T)},\qquad
\|\psi^\ell\|_{H^m(T^\ell)}\lesssim h_T^{-m} \|\psi^\ell\|_{L^2(T^\ell)}.
\end{split}
\end{equation}
\end{proposition}
\begin{proof}
The proof of the first trace inequality \eqref{tracineq}
follows from the standard trace inequality on affine faces
and mapping the result to $T$ using \eqref{eqn:NormEquivyLag}.
The second inequality in \eqref{tracineq}
follows from the first and an equivalence of norms in \eqref{eqn:NormEquivy}.
Likewise, the inverse inequalities \eqref{inveineq} follow
from the inverse inequality on affine faces and \eqref{eqn:NormEquivy}--\eqref{eqn:NormEquivyLag}.
\end{proof}

\begin{proposition}\label{prop:NormEquivalence}
There holds
\begin{align}
\sum_{T\in \calT_h} \|\nab^2_{\Gamma_h} \psi\|_{L^2(T)}^2 \lesssim  \|\psi\|_{2,h}^2,\quad \text{and}\quad
\sum_{T^\ell \in \calT_h^\ell} \|\nab^2_{\Gamma} \psi^\ell\|_{L^2(T^\ell)}^2 \lesssim  \|\psi^\ell\|_{2}^2\quad \forall \psi \in V_h.
\label{ineq.1}
\end{align}
Consequently, we have the following norm equivalences:
\begin{alignat}{2}
\label{NormCompar.1}
\|\psi\|_{2,h}\approx \tbar{\psi}_{2,h}\approx \tbar{\psi}_{H_h},\quad\text{and}\quad
\|\psi^\ell\|_{2}
\approx \tbar{\psi^\ell}_{2}\approx \tbar{\psi^\ell}_H
\qquad \forall \psi\in V_h.
\end{alignat}
\end{proposition}
\begin{proof}
The proof of \eqref{ineq.1} is given in Appendix \ref{app:NormE}. 
The norm equivalences given in \eqref{NormCompar.1} then follow
from \eqref{ineq.1} and the inverse and trace estimates in Proposition \ref{prop:Trace}.
\end{proof}

\begin{lemma} \label{lem:ContandCoer}
There holds for all $\psi,\chi\in W$,
\begin{align}
\label{eqn:ContEst}
\begin{aligned}
|a_h(\psi,\chi)|&\le (1+\sigma) \tbar{\psi}_{2,h}\tbar{\chi}_{2,h}\qquad &&\forall \psi,\chi\in W,\\
|a_h^\ell(\psi,\chi)|&\le (1+\sigma) \tbar{\psi}_{2}\tbar{\chi}_{2}\qquad &&\forall \psi,\chi\in W^\ell.
\end{aligned}
\end{align}
Moreover, for any $\alpha\in (0,1)$,
there exists $\sigma_0>0$ 
such that for $\sigma\ge \sigma_0$, there holds
\begin{equation}
\label{eqn:CoerciveOnGam}
\alpha \|\psi\|_{2,h}^2 \le a_h(\psi,\psi)\quad \forall \psi\in V_h,\quad 
\text{and}\quad
\alpha \|\psi^\ell\|_2^2 \le a_h^\ell(\psi^\ell,\psi^\ell)\quad \forall \psi\in V_h.
\end{equation}
Consequently, there exists a unique solution to \eqref{eqn:C0IPMethod} provided $\sigma$ is sufficiently large. \label{lem}
\end{lemma}

\begin{proof}
By 
the Cauchy-Schwarz inequality,
\begin{align*}
\left|a_{h}(\psi, \chi)\right| 
& \leq \sum_{T \in \calT_{h}}\left\|H_{\Gamma_{h}}(\psi)\right\|_{L^{2}(T)}\left\|H_{\Gamma_{h}}(\chi)\right\|_{L^{2}(T)} + \left\|\nab_{\Gamma_{h}}\psi\right\|_{L^{2}\left(\Gamma_{h}\right)}\left\|\nab_{\Gamma_{h}}\chi\right\|_{L^{2}\left(\Gamma_{h}\right)} \\
& \qquad + \sum_{e \in \calE_{h}}h_{e}^{\frac{1}{2}}\left\|\avg{\bt_h^{\intercal}H_{\Gamma_{h}}(\psi)\bmu_h}\right\|_{L^{2}(e)}h_{e}^{- \frac{1}{2}}\left\|\jump{\nab_{\Gamma_{h}}\chi}\right\|_{L^{2}(e)} \\
& \qquad + \sum_{e \in \calE_{h}}h_{e}^{\frac{1}{2}}\left\|\avg{\bt_h^{\intercal}H_{\Gamma_{h}}(\chi)\bmu_h}\right\|_{L^{2}(e)}h_{e}^{- \frac{1}{2}}\left\|\jump{\nab_{\Gamma_{h}}\psi}\right\|_{L^{2}(e)} \\
& \qquad + \sum_{e \in \calE_{h}}\left({\sigma h_e^{-1}}\right)^{\frac{1}{2}}\left\|\jump{\nab_{\Gamma_{h}}\psi}\right\|_{L^{2}(e)}\left({\sigma h_e^{-1}}\right)^{\frac{1}{2}}\left\|\jump{\nab_{\Gamma_{h}}\chi}\right\|_{L^{2}(e)} \\
%
& \leq (1+\sigma)\tbar{\psi}_{2, h}\tbar{\chi}_{2, h}.
\end{align*}
The second inequality in \eqref{eqn:ContEst} is proved similarly.

Next, by Proposition \ref{prop:Trace}, we find
\[
\sum_{e \in \calE_{h}}h_{e}\left\|\avg{\bt_h^{\intercal}H_{\Gamma_{h}}(\psi)\bmu_h}\right\|_{L^{2}(e)}^{2}
\leq C\sum_{T \in \calT_{h}}\left\|H_{\Gamma_{h}}(\psi)\right\|_{L^{2}(T)}^{2}\quad \forall \psi\in V_h
\]
for some constant $C>0$ independent of $h$.
Consequently, by the Cauchy-Schwarz inequality,
\begin{align*}
2\sum_{e \in \calE_{h}}\int_{e}\avg{\bt_h^{\intercal}H_{\Gamma_{h}}(\psi)\bmu_h}\jump{\nab_{\Gamma_{h}}\psi} 
& \leq \varepsilon\sum_{e \in \calE_{h}}h_{e}\left\|\avg{\bt_h^{\intercal}H_{\Gamma_{h}}(\psi)\bmu_h}\right\|_{L^{2}(e)}^{2} + \varepsilon^{-1}\sum_{e \in \calE_{h}}h_{e}^{- 1}\left\|\jump{\nab_{\Gamma_{h}}\psi}\right\|_{L^{2}(e)}^{2} \\
& \leq C\varepsilon\sum_{T \in \calT_{h}}\left\|H_{\Gamma_{h}}(\psi)\right\|_{L^{2}(T)}^{2} + \varepsilon^{-1} \sum_{e \in \calE_{h}}h_{e}^{- 1}\left\|\jump{\nab_{\Gamma_{h}}\psi}\right\|_{L^{2}(e)}^{2}
\end{align*}
for any $\varepsilon>0$.
Given $\alpha\in (0,1)$, we choose $1 - C\varepsilon = \alpha$ 
and take $\sigma_0  = \alpha+\varepsilon^{-1}$. Then for $\sigma\ge \sigma_0$,
\begin{align*}
a_{h}(\psi, \psi) & = \sum_{T \in \calT_{h}}\left\|H_{\Gamma_{h}}(\psi)\right\|_{L^{2}(T)}^2 + \left\|\nab_{\Gamma_{h}}\psi\right\|_{L^{2}\left(\Gamma_{h}\right)}^2 - 2\sum_{e \in \calE_{h}}\int_{e}\avg{\bt_h^{\intercal}H_{\Gamma_{h}}(\psi)\bmu_h}\jump{\nab_{\Gamma_{h}}\psi} \\
& \qquad + \sigma\sum_{e \in \calE_{h}}{h_{e}}^{-1}\left\|\jump{\nab_{\Gamma_{h}}\psi}\right\|_{L^{2}(e)}^{2} \\
& \geq \sum_{T \in \calT_{h}}\left\|H_{\Gamma_{h}}(\psi)\right\|_{L^{2}(T)}^2 + \alpha\left\|\nab_{\Gamma_{h}}\psi\right\|_{L^{2}\left(\Gamma_{h}\right)}^2 - C\varepsilon\sum_{T \in \calT_{h}}\left\|H_{\Gamma_{h}}(\psi)\right\|_{L^{2}(T)}^{2} \\
& \qquad - {\varepsilon}^{-1}\sum_{e \in \calE_{h}}h_{e}^{- 1}\left\|\jump{\nab_{\Gamma_{h}}\psi}\right\|_{L^{2}(e)}^{2} + \sigma\sum_{e \in \calE_{h}}{h_{e}}^{-1}\left\|\jump{\nab_{\Gamma_{h}}\psi}\right\|_{L^{2}(e)}^{2} \\
& \geq \alpha\|\psi\|_{2, h}^{2}.
\end{align*}

The second coercivity result in \eqref{eqn:CoerciveOnGam} follows from the exact same arguments,
but using the second inequalities in \eqref{tracineq}--\eqref{inveineq} instead of the first.
\end{proof}

\section{Geometric consistency estimates} \label{sec-Geo}
In this section, we derive estimates related to the geometric inconsistencies  of the scheme. As shown in the next section, these involve the disparities between the linear (source) forms $\ell(\cdot)$ and $\ell_h(\cdot)$, as well as the bilinear forms $a^\ell_h(\cdot, \cdot)$ and $a_h(\cdot, \cdot)$. An upper bound of the former is given in the next result. Its proof is found in Appendix \ref{app:DifferRelationsProof}.

{
\begin{lemma} \label{lem:SourceFormCons}
There holds for all $\psi\in V_h$, 
\begin{align}
\big|\ell(\psi^\ell) - \ell_h(\psi)\big| & \le \| {\bm f} - {\bm F}_h^\ell\|_{L^2(\Gamma)}\tbar{\psi^\ell}_H, \label{6.5}
\end{align}
where 
\begin{equation}
    \label{eqn:FhlDef}
    {\bm F}_h^\ell = ({\bf I}- d {\bf H})^{-1} \left({\bf I}- \frac{({\bf P}\bn_h)\otimes {({\bf P}_{h}\bn})}{\bn\cdot \bn_h}\right){\bm f}_h^\ell.
\end{equation}
\end{lemma}
}
\begin{remark}\label{rem:SourceBounds}
Note that, by \eqref{eqn:GeoBounds}, 
$\|{\bm F}_h^\ell- {\bm f}_h^\ell\|_{L^2(\Gamma)}\lesssim h^{k_g+1}\|{\bm f}_h^\ell\|_{L^2(\Gamma)}$, and therefore
$\|{\bm f}-{\bm F}_h^\ell\|_{L^2(\Gamma)}\lesssim 
\|{\bm f}-{\bm f}_h^\ell\|_{L^2(\Gamma)}+h^{k_g+1}\|{\bm f}_h^\ell\|_{L^2(\Gamma)}$.
\end{remark}

To prove analogous results for the discrete bilinear forms $a^\ell_h(\cdot, \cdot)$ and $a_h(\cdot, \cdot)$, we require two technical preliminary results. The proofs of the next two lemmas are given in Appendix   \ref{app:DifferRelationsProof}.

\begin{lemma} \label{lem:DifferRelations}
There holds for all $\psi\in W^\ell$,
\begin{align}
\|(H_\Gamma(\psi))^e - H_{\Gamma_h}(\psi^e)\|_{L^2(\calT_h)} & \lesssim h^{k_g}\tbar{\psi}_H, \label{ApproxOper.1} \\
\|(\nab_\Gamma \psi)^e - \nab_{\Gamma_h}\psi^e\|_{L^2(\calT_h)} & \lesssim h^{k_g}\tbar{\psi}_H, \label{ApproxOper.2} \\
h^{- \frac{1}{2}}\|(\jump{\nab_\Gamma \psi})^e - \jump{\nab_{\Gamma_h}\psi^e}\|_{L^2(\calE_h)} & \lesssim h^{k_g}\tbar{\psi}_H, \label{ApproxOper.3} \\
h^{\frac{1}{2}}\|(\avg{\bt^\intercal H_\Gamma(\psi)\bmu})^e - \avg{\bt^\intercal_h H_{\Gamma_h}(\psi^e)\bmu_h}\|_{L^2(\calE_h)} & \lesssim h^{k_g}\tbar{\psi}_H. \label{ApproxOper.4}
\end{align}
\end{lemma}

\begin{lemma} \label{lem:DifferRelationsAgain}
For $\psi, \chi \in H^3(\Gamma)$,
the following integral estimates hold:
\begin{align}
\big((H_\Gamma(\psi))^e, (H_\Gamma(\chi))^e - H_{\Gamma_h}(\chi^e)\big)_{\Gamma_h} & \lesssim h^{k_g + 1}\|\psi\|_{H^3(\Gamma)}\|\chi\|_{H^3(\Gamma)}, \label{F.1} \\
\big((\nab_\Gamma\psi)^e, (\nab_\Gamma\chi)^e - \nab_{\Gamma_h}\chi^e\big)_{\Gamma_h} & \lesssim h^{k_g + 1}\|\psi\|_{H^3(\Gamma)}\|\chi\|_{H^3(\Gamma)}, \label{F.2} \\
{\big((\avg{\bt^\intercal H_{\Gamma}(\psi)\bmu})^e, (\jump{\nab_\Gamma\chi})^e - \jump{\nab_{\Gamma_h}\chi^e}\big)_{\calE_h}} & {\lesssim {h^{2k_g-1}}\|\psi\|_{H^3(\Gamma)}\|\chi\|_{H^3(\Gamma)}. \label{F.3}}
\end{align}
\end{lemma}

\begin{lemma} \label{lem:StrongAGeo}
There holds for $\psi,\chi\in W$,
\begin{align}
\big|a_h^\ell(\psi^\ell, \chi^\ell) - a_{h}(\psi, \chi)\big| \lesssim h^{k_g}\tbar{\psi^\ell}_H\tbar{\chi^\ell}_H. \label{6.6}
\end{align}
\end{lemma}

\begin{proof}
The bilinear forms $a_h^\ell(\cdot, \cdot)$ and $a_h(\cdot, \cdot)$ each consists of four terms, and the estimate \eqref{6.6} is derived estimating each of them using Lemma \ref{lem:DifferRelations}. For example, using \eqref{ApproxOper.1} and \eqref{eqn:muCloseToOne}, we have
for $\psi,\chi\in W$,
\begin{align*}
& \sum_{T^{{\ell}} \in \calT_h^{{\ell}}}\int_{T^\ell}H_\Gamma(\psi^\ell) : H_\Gamma(\chi^\ell) - \sum_{T \in \calT_h}\int_{T}H_{\Gamma_h}(\psi) : H_{\Gamma_h}(\chi) \\
& \qquad = \sum_{T^{{\ell}} \in \calT_h^{{\ell}}}\int_{T^\ell}\Big(H_\Gamma(\psi^\ell) : H_\Gamma(\chi^\ell) - \mu_h^{- 1}(H_{\Gamma_h}(\psi))^\ell : (H_{\Gamma_h}(\chi))^\ell\Big) \\
& \qquad \le \sum_{T^{{\ell}} \in \calT_h^{{\ell}}}\Big(\|H_\Gamma(\psi^\ell) - \mu_h^{- 1}(H_{\Gamma_h}(\psi))^\ell\|_{L^2(T^\ell)}\|H_\Gamma(\chi^\ell)\|_{L^2(T^\ell)} \\
& \qquad \qquad + \|\mu_h^{- 1}(H_{\Gamma_h}(\psi))^\ell\|_{L^2(T^\ell)}\|(H_{\Gamma_h}(\chi))^\ell - H_\Gamma(\chi^\ell)\|_{L^2(T^\ell)}\Big) \lesssim h^{k_g}\tbar{\psi^\ell}_H\tbar{\chi^\ell}_H.
\end{align*}
The other terms are handled in the exact same way, but using \eqref{ApproxOper.2}--\eqref{ApproxOper.4} and \eqref{eqn:muCloseToOne}--\eqref{eqn:muECloseToOne}.
\end{proof}

\begin{lemma} \label{lem:WeakAGeo}
There holds for $\psi, \chi\in H^3(\Gamma)$,
\begin{align}
\big|a_h^\ell(\psi, \chi) - a_h(\psi^e, \chi^e)\big| \lesssim {\left(h^{k_g + 1}+h^{2k_g-1}\right)}\|\psi\|_{H^3(\Gamma)}\|\chi\|_{H^3(\Gamma)}. \label{F.31}
\end{align}
\end{lemma}

\begin{proof}
Similar to the proof of Lemma \ref{lem:StrongAGeo}, the bilinear forms $a_h^\ell(\cdot, \cdot)$ and $a_h(\cdot, \cdot)$ each consists of four terms, and the estimate \eqref{F.31} is obtained by utilizing Lemma \ref{lem:DifferRelationsAgain}. For example,
\begin{align*}
& \int_{\Gamma}H_\Gamma(\psi) : H_\Gamma(\chi) - \int_{\Gamma_h}H_{\Gamma_h}(\psi^e) : H_{\Gamma_h}(\chi^e) \\
& \qquad = \int_{\Gamma_h}(H_\Gamma(\psi))^e : (H_\Gamma(\chi))^e - \int_{\Gamma_h}H_{\Gamma_h}(\psi^e) : H_{\Gamma_h}(\chi^e) + \int_{\Gamma_h}(\mu_h - 1)(H_\Gamma(\psi))^e : (H_\Gamma(\chi))^e \\
%
%
& = \int_{\Gamma_h}\big((H_\Gamma(\psi))^e - H_{\Gamma_h}(\psi^e)\big) : (H_\Gamma(\chi))^e - \int_{\Gamma_h}(H_{\Gamma}(\psi))^e : \big(H_{\Gamma_h}(\chi^e) - (H_\Gamma(\chi))^e\big) \\
& \qquad - \int_{\Gamma_h}\big(H_{\Gamma_h}(\psi^e)-(H_\Gamma(\psi))^e\big) : \big(H_{\Gamma_h}(\chi^e) - (H_\Gamma(\chi))^e\big) + \int_{\Gamma_h}(\mu_h - 1)(H_\Gamma(\psi))^e : (H_\Gamma(\chi))^e.
\end{align*}
We then apply \eqref{eqn:muCloseToOne}, \eqref{F.1}, and \eqref{ApproxOper.1} to obtain
\begin{align*}
& \bigg|\int_{\Gamma}H_\Gamma(\psi) : H_\Gamma(\chi) - \int_{\Gamma_h}H_{\Gamma_h}(\psi^e) : H_{\Gamma_h}(\chi^e)\bigg| \\
& \qquad \lesssim h^{k_g + 1}\|\psi\|_{H^3(\Gamma)}\|\chi\|_{H^3(\Gamma)} + h^{2k_g}\tbar{\psi}_H\tbar{\chi}_H + 
h^{k_g + 1}\|\psi\|_{H^2(\Gamma)}\|\chi\|_{H^2(\Gamma)} \\
& \qquad \lesssim h^{k_g + 1}\|\psi\|_{H^3(\Gamma)}\|\chi\|_{H^3(\Gamma)}.
\end{align*}
The other terms in $\big|a_h^\ell(\psi, \chi) - a_h(\psi^e, \chi^e)\big|$ are handled similarly, but using \eqref{F.2}--\eqref{F.3} and \eqref{ApproxOper.2}--\eqref{ApproxOper.4}.

\end{proof}

\section{Convergence Analysis}\label{sec-converge}
Let $\bar \pi_h:C(\bar \Gamma_h)\to \bar V_h$ denote
the standard Lagrange nodal interpolant onto the space
of piecewise polynomials of degree $\le k$ with respect
to the affine mesh $\bar \calT_h$.
We then set $\pi_h:C(\Gamma_h)\to V_h$ 
to be the corresponding interpolant on the polynomial-mapped mesh,
defined such that 
$\pi_h \psi = \big(\bar \pi_h (\psi\circ \bp_{k_g})\big)\circ \bp_{k_g}^{-1}$.
Likewise, we define $\pi_h^\ell:C(\Gamma)\to V_h^\ell$
to be $\pi_h^\ell \psi = \big( \pi_h (\psi\circ \bp)\big)\circ \bp|_{\Gamma_h}^{-1}$, 
the lifted interpolant on $\Gamma$. This interpolant satisfies (cf.~\cite[(2.23)]{Demlow09})
\begin{align}
|\psi - \pi_{h}^\ell \psi|_{H^{m}(T^{\ell})} \lesssim h^{r - m}\|\psi\|_{H^{r}(T^\ell)}, \qquad  0\le m\le r\label{InterpolErr.2}
\end{align}
for all $\psi\in H^{r}(T^\ell)$ with $2\le r\le k+1$. Consequently, by the trace inequalities in 
Proposition \ref{prop:Trace} and 
the definition of $\tbar{\cdot}_2$ there holds
\begin{align}
\tbar{\psi - \pi_{h}^\ell \psi}_2 \lesssim h^{r - 2}\|\psi\|_{H^{r}(\Gamma)}\qquad \forall \psi\in H^{r}(\Gamma) \label{InterpolErr.1},
\end{align}
and $\tbar{\pi_h^\ell \psi}_H \lesssim \tbar{\pi_h^\ell \psi}_2\lesssim \|\pi_h^\ell \psi\|_2\lesssim \|\psi\|_{H^2(\Gamma)}$.

\begin{theorem}
Let $\phi$ be the exact solution to \eqref{eqn:streamfunction},
and let $\phi_h \in V_{h,0}$ solve the $C^0$ IP method \eqref{eqn:C0IPMethod}. If $\phi\in H^r(\Gamma)$ with $3\le r\le k+1$, then there holds
{
\begin{align}
\tbar{\phi - \phi_h^\ell}_2 \lesssim (h^{r - 2} + h^{k_g})\|\phi\|_{H^{r}(\Gamma)} + 
\|{\bm f}-{\bm F}_h^\ell\|_{L^2(\Gamma)},
\label{7.5}
\end{align}
}
where ${\bm F}_h^\ell$ is given by \eqref{eqn:FhlDef}.
\end{theorem}

\begin{proof}
Using the coercivity and continuity of $a_h^\ell(\cdot, \cdot)$ (cf.~Lemma \ref{lem:ContandCoer}) and Strang's lemma, we obtain
\[
\tbar{\phi - \phi_{h}^\ell}_2 \lesssim \inf_{\psi \in V_h}\Bigg(\tbar{\phi - \psi^\ell}_2 + \sup_{\chi \in V_h}\frac{\big|a_h^\ell(\psi^\ell, \chi^\ell) - a_h(\psi, \chi)\big|}{\tbar{\chi^\ell}_2}\Bigg) + \sup_{\chi \in V_h}\frac{\big|\ell(\chi^\ell) - \ell_h(\chi)\big|}{\tbar{\chi^\ell}_2}.
\]
Taking $\psi = \pi_h\phi^{{e}}$ in the infimum and applying \eqref{6.6}, \eqref{NormCompar.1}, \eqref{6.5}, \eqref{InterpolErr.1} and the stability bound $\|\pi_h^\ell\phi\|_2 \lesssim \|\phi\|_{H^2(\Gamma)}$ yields
\begin{align*}
\tbar{\phi - \phi_h^\ell}_2 & \lesssim \Bigg(\tbar{\phi - \pi_h^\ell\phi}_2 + \sup_{\chi \in V_h}\frac{\big|a_h^\ell(\pi_h^\ell\phi, \chi^\ell) - a_h(\pi_h\phi^{{e}}, \chi)\big|}{\tbar{\chi^\ell}_2}\Bigg) + \sup_{\chi \in V_h}\frac{\big|\ell(\chi^\ell) - \ell_h(\chi)\big|}{\tbar{\chi^\ell}_2} \\
& \lesssim \tbar{\phi - \pi_h^\ell\phi}_2 + h^{k_g}\tbar{\pi_h^\ell\phi}_H 
+\|{\bm f}-{\bm F}_h^\ell\|_{L^2(\Gamma)}\\
& \lesssim h^{r - 2}\|\phi\|_{H^r(\Gamma)} + h^{k_g}\|\phi\|_{H^2(\Gamma)} 
+\|{\bm f}-{\bm F}_h^\ell\|_{L^2(\Gamma)}.
\end{align*}
\end{proof}

Next, we derive error estimates in lower-order norms
using a duality argument. In order to do so,
we require the following elliptic regularity result.
Its proof is given in Appendix \ref{app-Regularity}.

\begin{lemma}\label{thm:RegBihar}
Let $m\in \{0,1\}$ and $\chi\in H^{-m}(\Gamma)$.
Let $u\in H^2(\Gamma)$ with $\int_\Gamma u=0$ satisfy
\begin{align*}
\frac12 \Delta_\Gamma^2 u + {\rm div}_\Gamma ((K-1)\nab_\Gamma u) & = \chi\qquad \text{on }\Gamma.
\end{align*}
Then there holds $u\in H^{4-m}(\Gamma)$ with $\|u\|_{H^{4-m}(\Gamma)}\lesssim \|\chi\|_{H^{-m}(\Gamma)}$.
\end{lemma}

{\begin{theorem} \label{Thm-7.2}
Let $\phi$ be the exact solution to \eqref{eqn:streamfunction} and let $\phi_h \in V_{h, 0}$ solve the $C^0$ IP method \eqref{eqn:C0IPMethod}. If $\phi\in H^r(\Gamma)$ with $3\le r\le k+1$, then there holds
for $m\in \{0,1\}$,
\begin{align*}
\big\|\nab_\Gamma^m(\phi - \phi_h^\ell)\big\|_{L^2(\Gamma)} 
&\lesssim \left(h^{r-\underline{m}}+h^{k_g+2-\underline{m}}+h^{k_g+r-2} +h^{2k_g-1}\right) \|\phi\|_{H^r(\Gamma)}
+ \|{\bm f}-{\bm F}_h^\ell\|_{L^2(\Gamma)},
\end{align*}
where $\underline{m}=m$ for $k\ge 3$,
and $\underline{m}=1$ for $k=2$.
In particular, if $\phi\in H^{k+1}(\Gamma)$ there holds
\begin{align*}
\|\nab_\Gamma^m(\phi-\phi_h^\ell)\|_{L^2(\Gamma)}
 &\lesssim \left\{
 \begin{array}{ll}
 \left(h^{2} + h^{k_g+1}+h^{2k_g-1}\right)\|\phi\|_{H^3(\Gamma)}
  + \|{\bm f} - {\bm F}_h^\ell\|_{L^2(\Gamma)}& \text{if }k=2,\bigskip\\
 \left(h^{k+1-m}+h^{k_g+2-m}
 +h^{2k_g-1}\right)\|\phi\|_{H^{k+1}(\Gamma)}
  +\|{\bm f} - {\bm F}_h^\ell\|_{L^2(\Gamma)}& \text{if }k\ge 3.
 \end{array}
 \right.
 \end{align*}

%
\end{theorem}
}
\begin{proof}
Let $u \in H^2(\Gamma)$ with $\int_\Gamma u = 0$ satisfy
\[
\frac12 \Delta_\Gamma^2 u + {\rm div}_{\Gamma}((K - 1)\nab_\Gamma u) = (-\Delta_\Gamma)^m (\phi - \phi_h^\ell)\in H^{-m}(\Gamma) \quad \text{on } \Gamma\quad (m=0,1),
\]
so that by Lemma \ref{thm:RegBihar}, $u\in H^{4-m}(\Gamma)$ with
\begin{equation} \label{eqn:EllipticReg}
\|u\|_{H^{4-m}(\Gamma)} \lesssim \|(\Delta_\Gamma)^m(\phi - \phi_h^\ell)\|_{H^{-m}(\Gamma)}\lesssim \|\nab_\Gamma^m (\phi-\phi_h^\ell)\|_{L^2(\Gamma)}.
\end{equation}
Note that, by the consistency of the scheme,
\[
a_h^\ell(\psi, u) = \big\langle\psi, (-\Delta_\Gamma)^m (\phi - \phi_h^\ell)\big\rangle_\Gamma \qquad \forall \psi \in W,
\]
and therefore
\begin{align}
\begin{split} \label{F.36}
& \|\nab_\Gamma^m(\phi - \phi_h^\ell)\|_{L^2(\Gamma)}^2\\  
&\qquad = a_h^\ell\big(\phi - \phi_h^\ell, u\big) \\
&\qquad = a_h^\ell\big(\phi - \phi_h^\ell, u - \pi_h^\ell u\big) + a_h^\ell\big(\phi - \phi_h^\ell, \pi_h^\ell u\big) \\
&\qquad = a_h^\ell\big(\phi - \phi_h^\ell, u - \pi_h^\ell u\big) + \big[\ell(\pi_h^\ell u) - \ell_h(\pi_h u^{{e}})\big] + \big[a_h(\phi_h, \pi_h u^{{e}}) - a_h^\ell(\phi_h^\ell, \pi_h^\ell u)\big] \\
&\qquad =: I_1 + I_2 + I_3.
\end{split}
\end{align}
Let us estimate each term in \eqref{F.36} separately. First, 
note that
\begin{align}
\tbar{u - \pi_h^\ell u}_2 \lesssim h^{2-\underline{m}}\|u\|_{H^{4-m}(\Gamma)} \lesssim h^{2-\underline{m}} \|\nab_\Gamma^m(\phi - \phi_h^\ell)\|_{L^2(\Gamma)}. \label{E.3}
\end{align}
 Then, by the continuity estimate of $a_h^\ell(\cdot, \cdot)$ in Lemma \ref{lem:ContandCoer} and \eqref{7.5}, there holds
 \begin{align}
 \begin{split}
 I_1 
 & = a_h^\ell\big(\phi - \phi_h^\ell, u - \pi_h^\ell u\big) \\
 & \lesssim \tbar{\phi - \phi_h^\ell}_2\tbar{u - \pi_h^\ell u}_2 \\
 & \lesssim \Big((h^{r -\underline{m}} + h^{k_g+2-\underline{m}})\|\phi\|_{H^{r}(\Gamma)} 
 + h^{2-\underline{m}}\|{\bm f} - {\bm F}_h^\ell\|_{L^2(\Gamma)}\Big) 
 \|\nab_\Gamma^m(\phi-\phi_h^\ell)\|_{L^2(\Gamma)}.
 \end{split} \label{F.38}
 \end{align}
 Next, due to \eqref{6.5}, the stability bound $\tbar{\pi_h^\ell u}_H \lesssim \|u\|_{H^2(\Gamma)}$ and \eqref{eqn:EllipticReg},
 \begin{align}
 \begin{split}
 I_2 
 & = \ell(\pi_h^\ell u) - \ell_h(\pi_h u^{{e}}) 
  \lesssim \|{\bm f}- {\bm F}_h^\ell \|_{L^2(\Gamma)}
 \tbar{\pi_h^\ell u}_H 
  \lesssim \|{\bm f} - {\bm F}_h^\ell\|_{L^2(\Gamma)}
 \|\nab_\Gamma^m(\phi - \phi_h^\ell)\|_{L^2(\Gamma)}.
 \end{split} \label{F.40}
 \end{align}
 Now, we consider the following in \eqref{F.36},
 \begin{align}
 \begin{split}
 I_3 & = a_h(\phi_h, \pi_h u^{{e}}) - a_h^\ell(\phi_h^\ell, \pi_h^\ell u) \\
 & = \big[a_h\big(\phi_h, \pi_h u^{{e}} - u^e\big) - a_h^\ell\big(\phi_h^\ell, \pi_h^\ell u - u\big)\big] + \big[a_h\big(\phi_h - \phi^e, u^e\big) - a_h^\ell\big(\phi_h^\ell - \phi, u\big)\big] \\
  & \qquad + \big[a_h(\phi^e, u^e) - a_h^\ell(\phi, u)\big] \\
 & =: I_{3,1} + I_{3,2} + I_{3,3}.
 \end{split} \label{F.41}
 \end{align}

 \textit{Estimate $I_{3,1}$}: Applying \eqref{NormCompar.1}, \eqref{InterpolErr.1} and \eqref{7.5} yields 
 \begin{align}
 \begin{split}
  \tbar{\phi - \phi_h^\ell}_H
 & \leq \tbar{\phi - \pi_h^\ell\phi}_H + \tbar{\pi_h^\ell\phi - \phi_h^\ell}_H \\
 & \lesssim h^{r - 2}\|\phi\|_{H^{r}(\Gamma)} + \tbar{\pi_h^\ell\phi - \phi}_2 + \tbar{\phi - \phi_h^\ell}_2 \\
 & \lesssim (h^{r - 2} + h^{k_g})\|\phi\|_{H^{r}(\Gamma)} 
 + \|{\bm f} - {\bm F}_h^\ell\|_{L^2(\Gamma)}.
 \end{split} \label{F.42}
 \end{align}
 Therefore, $\tbar{\phi_h^\ell}_H\lesssim \|\phi\|_{H^r(\Gamma)}+\|{\bm f}-{\bm F}_h^\ell\|_{L^2(\Gamma)}$, and so by \eqref{6.6}, \eqref{NormCompar.1} and \eqref{E.3},
 \begin{align}
 \begin{split}
 I_{3,1} & = a_h\big(\phi_h, \pi_h u^{{e}} - u^e\big) - a_h^\ell\big(\phi_h^\ell, \pi_h^\ell u - u\big) \\
 & \lesssim h^{k_g}\tbar{\phi_h^\ell}_H\tbar{\pi_h^\ell u - u}_H \\
 &\lesssim h^{k_g+2-\underline{m}} \left(\|\phi\|_{H^r(\Gamma)}
 +\|{\bm f}-{\bm F}_h^\ell\|_{L^2(\Gamma)}\right) \|\nab_\Gamma^m(\phi-\phi_h^\ell)\|_{L^2(\Gamma)}.
 \end{split} \label{F.43}
 \end{align}

 \textit{Estimate $I_{3,2}$}: Using \eqref{6.6}, \eqref{F.42}, \eqref{NormCompar.1}, and \eqref{eqn:EllipticReg}, 
 \begin{align}
 \begin{split}
 I_{3,2} 
 & = a_h\big(\phi_h - \phi^e, u^e\big) - a_h^\ell\big(\phi_h^\ell - \phi, u\big) \\
 & \lesssim h^{k_g}\tbar{\phi_h^\ell - \phi}_H\tbar{u}_H \\
 & \lesssim h^{k_g}\Big((h^{r - 2} + h^{k_g})\|\phi\|_{H^{r}(\Gamma)} 
 + \|{\bm f} - {\bm F}_h^\ell\|_{L^2(\Gamma)}\Big)\|\nab_\Gamma^m(\phi - \phi_h^\ell)\|_{L^2(\Gamma)}.
 \end{split} \label{F.44}
 \end{align}

\textit{Estimate $I_{3,3}$}: By \eqref{F.31} and \eqref{eqn:EllipticReg}, 
{ \begin{align}
 \begin{split}
 I_{3,3} 
 & = a_h(\phi^e, u^e) - a_h^\ell(\phi, u) \\
  & \lesssim (h^{2k_g -1}+h^{k_g+1})\|\phi\|_{H^3(\Gamma)}\|u\|_{H^3(\Gamma)} \\
&\lesssim (h^{2k_g - 1}+h^{k_g+1})\|\phi\|_{H^r(\Gamma)}\|\nab^m_\Gamma(\phi-\phi_h^\ell)\|_{L^2(\Gamma)}.
 \end{split} \label{F.45}
 \end{align}
 }

 Combining \eqref{F.41}, \eqref{F.43}, \eqref{F.44} and \eqref{F.45} yields
\begin{align}
 \begin{split}
 I_3
 &\lesssim \Bigg( 
\left(h^{k_g+2-\underline{m}} + h^{r-2+k_g}+h^{2k_g}
+h^{2k_g-1}+h^{k_g+1}\right)
\|\phi\|_{H^r(\Gamma)}\\
 &\qquad + \left( h^{k_g+2-\underline{m}} + h^{k_g}\right) \|{\bm f}-{\bm F}_h^\ell\|_{L^2(\Gamma)}\Bigg)\|\nab_\Gamma^m(\phi-\phi_h^\ell)\|_{L^2(\Gamma)}\\
  &\lesssim 
h^{k_g}\left( \left(h^{2-\underline{m}} + h^{r-2}
 +h^{k_g-1}\right)
 \|\phi\|_{H^r(\Gamma)}
 +  \|{\bm f}-{\bm F}_h^\ell\|_{L^2(\Gamma)}\right)
    \|\nab_\Gamma^m(\phi-\phi_h^\ell)\|_{L^2(\Gamma)}.
 \end{split} \label{E.11}
 \end{align}

 It follows from \eqref{F.36}, \eqref{F.38}, \eqref{F.40} and \eqref{E.11} that
\begin{align*}
\|\nab_\Gamma^m(\phi-\phi_h^\ell)\|_{L^2(\Gamma)}
&\lesssim \left(h^{r-\underline{m}}+h^{k_g+2-\underline{m}}+h^{k_g+2-\underline{m}}+h^{k_g+r-2} +h^{2k_g-1}\right)\|\phi\|_{H^r(\Gamma)}\\
%
&\qquad+ \left(h^{2-\underline{m}}+1+h^{k_g}\right)\|{\bm f}-{\bm F}_h^\ell\|_{L^2(\Gamma)}\\
&\lesssim \left(h^{r-\underline{m}}+h^{k_g+2-\underline{m}}+h^{k_g+r-2} +h^{2k_g-1}\right) \|\phi\|_{H^r(\Gamma)}
+ \|{\bm f}-{\bm F}_h^\ell\|_{L^2(\Gamma)}
 \end{align*}

\end{proof}

\begin{corollary} \label{Cor-6.4}
Let $\phi$ be the exact solution to \eqref{eqn:streamfunction} and let $\phi_h \in V_{h, 0}$ solve the $C^0$ IP method \eqref{eqn:C0IPMethod}. 
Define $\bu := \bcurl_\Gamma\phi$ and $\bu_h := \bcurl_{\Gamma_h}\phi_h$. If $\phi \in H^{k + 1}(\Gamma)$, then 
\[
\|\bu - \calP_\bp\bu_h\|_{L^2(\Gamma)} \lesssim (h^k + h^{k_g + 1} + h^{2k_g - 1})\|\phi\|_{H^{k + 1}(\Gamma)} + \|{\bm f} - {\bm F}_h^\ell\|_{L^2(\Gamma)},
\]
where 
\[
\calP_{\bp}\bu_h \circ \bp = \mu_h^{-1} ({\bf P}-d {\bf H})\bu_h
\]
is the Piola transform of $\bu_h$ with respect to the closest point projection $\bp$.
\end{corollary}
\begin{proof}
Using \eqref{eqn:curlChain} below, we can conclude $\calP_{\bp}\bu_h = \calP_{\bp} (\bcurl_{\Gamma_h} \phi_h) = \bcurl_{\Gamma} \phi_h^\ell$.
Therefore by applying the result from the particular case $m = 1$ in Theorem \ref{Thm-7.2} yields
\begin{align*}
\|\bu - \calP_\bp\bu_h\|_{L^2(\Gamma)}
& = \|\bcurl_\Gamma\phi - \bcurl_\Gamma\phi_h^\ell\|_{L^2(\Gamma)} \\
& = \|\bn \times \nab_\Gamma(\phi - \phi_h^\ell)\|_{L^2(\Gamma)} \\
& \lesssim (h^k + h^{k_g + 1} + h^{2k_g - 1})\|\phi\|_{H^{k + 1}(\Gamma)} + \|{\bm f} - {\bm F}_h^\ell\|_{L^2(\Gamma)}.
\end{align*}
\end{proof}

\section{Numerical Experiments}\label{sec-numerics}

In the numerical experiments, 
we consider an ellipsoid $\Gamma = \{x \in \mathbb{R}^3: \Psi(x) = 0:\ \Psi(x) = x_1^2 + x_2^2 / 2 + x_3^2 / 2 - 1\}$,
and take the stream function solution as $\phi(x) = e^{x_1}(\cos(x_2) + x_3)$
and the pressure solution as $p(x) = x_1x_2x_3$. 
We determine the velocity $\bu = {\bf curl}_{\Gamma}\phi$ and the force ${\bm f}$ on the 
right hand side of \eqref{eqn:momentum} through  Mathematica.
These expressions for $\bu$ and ${\bm f}$ are well-defined
on $\bbR^3$, and we respectively take $\bu^{ce}$ and ${\bm f}^{ce}$ to be these
expressions defined on $\Gamma_h$. In all of the numerical experiments, we take ${\bm f}_h = {\bm f}^{ce}$,
and note that (cf.~Lemma \ref{lem:SourceFormCons}
and Remark \ref{rem:SourceBounds})
\begin{equation}\label{eqn:SourceRate}
\|{\bm f}-{\bm F}_h^\ell\|_{L^2(\Gamma)}\lesssim h^{k_g+1}.
\end{equation}

Using NGSolve \cite{ngsolve}, we solve the $C^0$ IP method \eqref{eqn:C0IPMethod} for various values of $k$ and $k_g$ while varying $h$,
and also compute the approximate velocity function $\bu_h = \bcurl_{\Gamma_h} \phi_h$.
We compute the errors on $\Gamma_h$, and therefore we need to map both the exact stream function $\phi$ and velocity $\bu$ 
to the discrete surface $\Gamma_h$. Ideally, this would be accomplished via the
closest point project and the Piola transform with respect to the inverse of the closest point projection, respectively:
\[
\phi^e = \phi \circ \bp,\qquad
\mathcal{P}_{\bp^{- 1}}\bu = \mu_h\bigg({\bf I} - \frac{\bn \otimes \bn_h}{\bn \cdot \bn_h}\bigg)({\bf I} - d{\bf H})^{- 1}(\bu \circ \bp).
\]
Unfortunately, an explicit formula for the distance function of the ellipsoid is unknown,
and therefore its closest point projection
is unavailable. Instead, based on \cite[Section 4.2]{DemlowDziuk07}, we make the approximations
\[
\tilde \bn(x) = \frac{\nab \Psi(x)}{|\nab \Psi(x)|},\quad \tilde d(x) = \frac{\Psi(x)}{|\nab \Psi(x)|},\quad \tilde{\bp}(x) = x - \tilde d(x) \tilde \bn(x),\qquad x\in \Gamma_h.
\]
We then have $|d(x)-\tilde d(x)| = O(d) = O(h^{k_g+1})$, and
\begin{align*}
|\bn(x) - \tilde \bn(x)| = \left| \frac{\nab \Psi(\bp(x))}{|\nab \Psi(\bp(x))|} - \frac{\nab \Psi(x)}{|\nab \Psi(x)|}\right|
\lesssim |\bp(x)-x| = O(h^{k_g+1}).
\end{align*}
Therefore $|\bp(x)-\tilde \bp(x)| = O(h^{k_g+1})$, although $\tilde \bp(x)$ does not necessarily lie on $\Gamma$.

Similar to above, we let
\[
\phi^{ce}(x) = e^{x_1} (\cos(x_2)+x_3),\quad x\in \Gamma_h,
\]
be the canonical extension of $\phi$
and set
\[
\tilde \phi^{e}  = \phi^{ce} \circ \tilde \bp,
\]
as the approximate extension to $\phi$.
These extensions satisfy $|\phi^e(x) - \tilde \phi^e(x)|\lesssim h^{k_g+1}$
and  $|\phi^e(x)- \phi^{ce}(x)|\lesssim h^{k_g+1}$.
Therefore, Theorem \ref{Thm-7.2} with $m=0$ and \eqref{eqn:SourceRate} yield 
\begin{align}\label{eqn:PhiExpect}
\|\tilde \phi^{e}-\phi_h\|_{L^2(\Gamma_h)}\lesssim 
\left\{
\begin{array}{ll}
h^2+h^{k_g+1} + h^{2k_g-1} & k=2\\
h^{k+1}+h^{k_g+1} + h^{2k_g-1} & k\ge 3
\end{array},
\right.
\end{align}
and the same orders of convergence hold for $\| \phi^{ce}-\phi_h\|_{L^2(\Gamma_h)}$.

We map the exact velocity $\bu$ to $\Gamma_h$ using
an approximate Piola transform
\[
\tilde{\calP}_{\bp^{- 1}}\bu := \left({\bf I} - \frac{(\tilde \bn \circ \tilde{\bp}) \otimes \bn_h}{(\tilde \bn \circ \tilde{\bp}) \cdot \bn_h}\right)(\bu^{ce} \circ \tilde{\bp}).
\]
Note that $|\tilde{\calP}_{\bp^{- 1}}\bu - \calP_{\bp^{- 1}}\bu| = O(d) = O(h^{k_g + 1})$,
and so $\|\tilde{\calP}_{\bp^{- 1}}\bu - \bu_{h}\|_{L^2(\Gamma_h)}$ 
has the same order of convergence as stated in Corollary \ref{Cor-6.4}:
\begin{equation}\label{eqn:UExpect}
\|\tilde \calP_{\bp^{-1}} \bu - \bu_h\|_{L^2(\Gamma)}\lesssim h^k+h^{k_g+1}+h^{2k_g-1}.
\end{equation}
However, $|\bu^{ce}-\calP_{\bp^{-1}} \bu| =O(h^{k_g})$ on $\Gamma_h$,
and therefore Corollary \ref{Cor-6.4} only yields
\begin{equation}\label{eqn:UEExpect}
\|\bu^{ce} - \bu_h\|_{L^2(\Gamma)}\lesssim h^k+h^{k_g}.
\end{equation}

The errors for $\|\tilde \phi^e - \phi_{h}\|_{L^{2}(\Gamma_{h})}$
and $\|\tilde \calP_{\bp^{-1}} \bu - \bu_h\|_{L^2(\Gamma_{h})}$
are provided in Table \ref{tab:Numerics}.
For the velocity error, we see rates of convergence
that are in agreement with \eqref{eqn:UExpect},
except in the case $k_g=1$, where we observe
second-order convergence, instead of first order.
This suggests that the term $h^{2k_g-1}$ in \eqref{eqn:UExpect}
leads to an estimate that is not sharp.

For the stream function error, we see rates of convergence of order $O(h^{k+1})$
for $k\ge 3$ in the isoparametric setting $k_g=k$, which is in agreement
with the theoretical estimate \eqref{eqn:PhiExpect}. 
However, we also observe the optimal order $O(h^{4})$
in the case $k=3$ and $k_g=2$,
and the  order $O(h^4)$ for $k=4$ and $k_g=2$.
Similar to the velocity error, we also observe second-order convergence
in the case $k_g=1$.
The rates of convergence for $k=4$ and $k_g\in \{3,4\}$
are unclear, as we observe a degradation of the error,
possibly due to round-off error and the poor conditioning
of the system.

Next, we list the errors 
$\|\phi^{ce} - \phi_{h}\|_{L^{2}(\Gamma_{h})}$
and $\|\bu^{ce} - \bu_h\|_{L^2(\Gamma_{2{h}})}$
in Table \ref{tab:Numerics2}.
Here, we see exact agreement 
with the expected rates of convergence 
of the velocity given in \eqref{eqn:UEExpect}.
Likewise, the stream function error convergence
rates conform to the theoretical estimate \eqref{eqn:PhiExpect},
except in the case $k_g=1$, where we observe second-order convergence.

\begin{table}[h]
\centering
\renewcommand\thetable{1}
\caption{\label{tab:Numerics} Convergence 
rates of the computed stream function $\phi_h$ and velocity $\bu_h$.
We list the theoretical rates of convergence 
\eqref{eqn:PhiExpect} and \eqref{eqn:UExpect} in blue.}
\begin{tabular}{c c c c c c c}
\hline
$k$ & $k_g$ & $h$ & $\| \tilde \phi^{e} - \phi_{h}\|_{L^{2}(\Gamma_{h})}$ & Rate & $\|\tilde{\calP}_{\bp^{- 1}}\bu - \bu_{h}\|_{L^2(\Gamma_h)}$ & Rate \\
\hline
\multirow{8}{*}{2} & \multirow{4}{*}{1} & 0.2 & 0.05767 & \textcolor{blue}{(1)} & 0.07858 & \textcolor{blue}{(1)} \\
 &  & 0.1 & 0.01424 & 2.01775 & 0.02016 & 1.96237 \\
 &  & 0.05 & 0.00376 & 1.91772 & 0.00531 & 1.92421 \\
 &  & 0.025 & 0.00093 & 2.00579 & 0.00131 & 2.01714 \\
\cline{2-7}
 & \multirow{4}{*}{2} & 0.2 & 0.01297 & \textcolor{blue}{(2)} & 0.03828 & \textcolor{blue}{(2)} \\
 &  & 0.1 & 0.00353 & 1.87457 & 0.01015 & 1.91519 \\
 &  & 0.05 & 0.00097 & 1.85711 & 0.00279 & 1.86280 \\
 &  & 0.025 & 0.00022 & 2.10957 & 0.00066 & 2.07626 \\
\hline
\multirow{12}{*}{3} & \multirow{4}{*}{1} & 0.2 & 0.08097 & \textcolor{blue}{(1)} & 0.08958 & \textcolor{blue}{(1)} \\
 &  & 0.1 & 0.02028 & 1.99731 & 0.02264 & 1.98393 \\
 &  & 0.05 & 0.00547 & 1.88983 & 0.00611 & 1.88769 \\
 &  & 0.025 & 0.00135 & 2.01742 & 0.00151 & 2.01603 \\
\cline{2-7}
 & \multirow{4}{*}{2} & 0.2 & 8.47941e-05 & \textcolor{blue}{(3)} & 0.00094 & \textcolor{blue}{(3)} \\
 &  & 0.1 & 6.31325e-06 & 3.74751 & 0.00012 & 2.90593 \\
 &  & 0.05 & 3.63872e-07 & 4.11687 & 1.49323e-05 & 3.07424 \\
 &  & 0.025 & 2.70954e-08 & 3.74731 & 1.84384e-06 & 3.01764 \\
\cline{2-7}
 & \multirow{4}{*}{3} & 0.2 & 0.00039 & \textcolor{blue}{(4)} & 0.00290 & \textcolor{blue}{(3)} \\
 &  & 0.1 & 2.88117e-05 & 3.78754 & 0.00037 & 2.94548 \\
 &  & 0.05 & 2.04527e-06 & 3.81629 & 4.91400e-05 & 2.94155 \\
 &  & 0.025 & 1.27143e-07 & 4.00776 & 6.14542e-06 & 2.99931 \\
\hline
\multirow{16}{*}{4} & \multirow{4}{*}{1} & 0.2 & 0.07802 & \textcolor{blue}{(1)} & 0.08622 & \textcolor{blue}{(1)} \\
 &  & 0.1 & 0.01955 & 1.99660 & 0.02179 & 1.98379 \\
 &  & 0.05 & 0.00526 & 1.89306 & 0.00587 & 1.89232 \\
 &  & 0.025 & 0.00130 & 2.01383 & 0.00145 & 2.01115 \\
\cline{2-7}
 & \multirow{4}{*}{2} & 0.2 & 6.40078e-05 & \textcolor{blue}{(3)} & 0.00015 & \textcolor{blue}{(3)} \\
 &  & 0.1 & 5.45770e-06 & 3.55188 & 1.73855e-05 & 3.19389 \\
 &  & 0.05 & 3.08108e-07 & 4.14678 & 2.13295e-06 & 3.02696 \\
 &  & 0.025 & 1.78093e-08 & 4.11273 & 2.41951e-07 & 3.14006 \\
\cline{2-7}
 & \multirow{4}{*}{3} & 0.2 & 7.94347e-06 & \textcolor{blue}{(4)} & 0.00014 & \textcolor{blue}{(4)} \\
 &  & 0.1 & 5.70919e-07 & 3.79841 & 1.13606e-05 & 3.68062 \\
 &  & 0.05 & 3.04503e-08 & 4.22875 & 6.80985e-07 & 4.06027 \\
 &  & 0.025 & 5.03957e-09 & 2.59508 & 4.27816e-08 & 3.99255 \\
\cline{2-7}
 & \multirow{4}{*}{4} & 0.2 & 4.12926e-06 & \textcolor{blue}{(5)} & 0.00011 & \textcolor{blue}{(4)} \\
 &  & 0.1 & 1.64135e-07 & 4.65292 & 8.43260e-06 & 3.73177 \\
 &  & 0.05 & 4.95913e-09 & 5.04865 & 5.12438e-07 & 4.04052 \\
 &  & 0.025 & 1.19892e-08 & {-1.27358} & 3.45115e-08 & 3.89222 \\
\hline
\end{tabular}
\end{table}

\begin{table}[h]
\centering
\renewcommand\thetable{2}
\caption{\label{tab:Numerics2} Convergence for the stream function $\phi_h$ and the velocity $\bu_h$.
We list the theoretical rates of convergence 
\eqref{eqn:PhiExpect} and \eqref{eqn:UEExpect} in blue.}
\begin{tabular}{c c c c c c c}
\hline
$k$ & $k_g$ & $h$ & $\|\phi^{ce} - \phi_{h}\|_{L^{2}(\Gamma_{h})}$ & Rate & $\|\bu^{ce} - \bu_{h}\|_{L^2(\Gamma_h)}$ & Rate \\
\hline
\multirow{8}{*}{2} & \multirow{4}{*}{1} & 0.2 & 0.03561 & \textcolor{blue}{(1)} & 0.24536 & \textcolor{blue}{(1)} \\
 &  & 0.1 & 0.00851 & 2.06441 & 0.12290 & 0.99739 \\
 &  & 0.05 & 0.00231 & 1.88026 & 0.06283 & 0.96780 \\
 &  & 0.025 & 0.00057 & 2.00569 & 0.03075 & 1.03081 \\
\cline{2-7}
 & \multirow{4}{*}{2} & 0.2 & 0.01295 & \textcolor{blue}{(2)} & 0.03845 & \textcolor{blue}{(2)} \\
 &  & 0.1 & 0.00353 & 1.87321 & 0.01019 & 1.91520 \\
 &  & 0.05 & 0.00097 & 1.85674 & 0.00280 & 1.86293 \\
 &  & 0.025 & 0.00022 & 2.10950 & 0.00066 & 2.07623 \\
\hline
\multirow{12}{*}{3} & \multirow{4}{*}{1} & 0.2 & 0.06046 & \textcolor{blue}{(1)} & 0.24972 & \textcolor{blue}{(1)} \\
 &  & 0.1 & 0.01490 & 2.02080 & 0.12343 & 1.01652 \\
 &  & 0.05 & 0.00411 & 1.85740 & 0.06292 & 0.97214 \\
 &  & 0.025 & 0.00101 & 2.02480 & 0.03076 & 1.03225 \\
\cline{2-7}
 & \multirow{4}{*}{2} & 0.2 & 0.00014 & \textcolor{blue}{(3)} & 0.00406 & \textcolor{blue}{(2)} \\
 &  & 0.1 & 1.80827e-05 & 3.04799 & 0.00098 & 2.04314 \\
 &  & 0.05 & 2.29863e-06 & 2.97576 & 0.00026 & 1.91730 \\
 &  & 0.025 & 2.77898e-07 & 3.04814 & 6.16448e-05 & 2.08103 \\
\cline{2-7}
 & \multirow{4}{*}{3} & 0.2 & 0.00040 & \textcolor{blue}{(4)} & 0.00310 & \textcolor{blue}{(3)} \\
 &  & 0.1 & 2.90516e-05 & 3.78777 & 0.00040 & 2.93645 \\
 &  & 0.05 & 2.06015e-06 & 3.81779 & 5.32761e-05 & 2.92893 \\
 &  & 0.025 & 1.28190e-07 & 4.00638 & 6.62167e-06 & 3.00822 \\
\hline
\multirow{16}{*}{4} & \multirow{4}{*}{1} & 0.2 & 0.05748 & \textcolor{blue}{(1)} & 0.24858 & \textcolor{blue}{(1)} \\
 &  & 0.1 & 0.01416 & 2.02110 & 0.12329 & 1.01163 \\
 &  & 0.05 & 0.00390 & 1.86000 & 0.06290 & 0.97093 \\
 &  & 0.025 & 0.00096 & 2.02019 & 0.03076 & 1.03186 \\
\cline{2-7}
 & \multirow{4}{*}{2} & 0.2 & 0.00015 & \textcolor{blue}{(3)} & 0.00395 & \textcolor{blue}{(2)} \\
 &  & 0.1 & 1.85841e-05 & 3.04004 & 0.00097 & 2.01542 \\
 &  & 0.05 & 2.30884e-06 & 3.00883 & 0.00026 & 1.90815 \\
 &  & 0.025 & 2.77587e-07 & 3.5615 & 6.16184e-05 & 2.07937 \\
\cline{2-7}
 & \multirow{4}{*}{3} & 0.2 & 1.87417e-05 & \textcolor{blue}{(4)} & 0.00109 & \textcolor{blue}{(3)} \\
 &  & 0.1 & 1.43958e-06 & 3.70253 & 0.00014 & 2.88171 \\
 &  & 0.05 & 8.56107e-08 & 4.07171 & 2.05907e-05 & 2.85533 \\
 &  & 0.025 & 6.32045e-09 & 3.75969 & 2.46611e-06 & 3.06168 \\
\cline{2-7}
 & \multirow{4}{*}{4} & 0.2 & 4.21651e-06 & \textcolor{blue}{(5)} & 0.00011 & \textcolor{blue}{(4)} \\
 &  & 0.1 & 1.67527e-07 & 4.65358 & 8.62432e-06 & 3.73354 \\
 &  & 0.05 & 5.10129e-09 & 5.03738 & 5.29696e-07 & 4.02517 \\
 &  & 0.025 & 1.19894e-08 & {-1.23282} & 3.52220e-08 & 3.91061 \\
\hline
\end{tabular}
\end{table}

\bibliographystyle{siam}
\bibliography{ref}

\appendix

\section{Proof of Lemma \ref{lem:HessForm}}\label{app:ProofHessForm}
We start with applying \eqref{eqn:NotSymmetric} to obtain the identity
\[
\underline{D}_i \underline{D}_j^2 \phi = \underline{D}_j \underline{D}_i \underline{D}_j
\phi+ n_i ({\bf H} \nab_\Gamma \underline{D}_j \phi)_j -n_j ({\bf H}\nab_\Gamma \underline{D}_j\phi)_i\quad i,j=1,2,3.
\]
Furthermore, \eqref{eqn:NotSymmetric} and $\underline{D}_{{\Gamma}}^2 \phi \bn = 0$ also yields
$\bn^\intercal \underline{D}_{{\Gamma}}^2 \phi = -\nab_\Gamma \phi^\intercal {\bf H}$.
Using these two identities and Lemma \ref{lem:IBP} then get
 \begin{align*}
\int_S (\Delta_\Gamma^2 \phi) \psi 
&= -\sum_{i,j=1}^3\int_S \big(\underline{D}_j \underline{D}_i \underline{D}_j
 \phi+ n_i ({\bf H} \nab_\Gamma \underline{D}_j\phi)_j -n_j ({\bf H}\nab_\Gamma \underline{D}_j\phi)_i\big) \underline{D}_i \psi
+\int_{\p S} (\nab_\Gamma \Delta_\Gamma \phi\cdot \bmu_S)\psi\\
  & = \sum_{i,j=1}^3 \left( \int_S (\underline{D}_i \underline{D}_j \phi)(\underline{D}_j \underline{D}_i \psi)
  -\int_S (\underline{D}_i \underline{D}_j \phi) \underline{D}_i \psi {\rm tr}({\bf H})n_j 
  -\int_{\p S} (\underline{D}_i \underline{D}_j \phi) \underline{D}_i \psi(\mu_S)_j \right)\\
 &\qquad +\sum_{i,j=1}^3 \int_S n_j ({\bf H}\nab_\Gamma \underline{D}_j\phi)_i \underline{D}_i \psi+\int_{\p S} (\nab_\Gamma \Delta_\Gamma \phi\cdot \bmu_S)\psi\\
 & = \int_S \hessone_{{\Gamma}} \phi:\hessone_{{\Gamma}} \psi^\intercal
 +\int_S \nab_\Gamma \phi^\intercal ( {\rm tr}({\bf H}) {\bf H}-{\bf H}^2) \nab_\Gamma \psi
  \\
 &\qquad  +\int_{\p S} (\nab_\Gamma \Delta_\Gamma \phi\cdot \bmu_S)\psi
 -\int_{\p S} \bmu^\intercal_S \underline{D}_{{\Gamma}}^2 \phi \nab_\Gamma \psi.
 \end{align*}
Continuing, we use \eqref{eqn:NotSymmetric} and the identity 
$\bn^\intercal \underline{D}_{{\Gamma}}^2 \phi = -\nab_\Gamma \phi^\intercal {\bf H}$
once again to obtain
\begin{align*}
        \int_S (\Delta_\Gamma^2 \phi) \psi 
&= \int_S ( \hessone_{{\Gamma}} \phi: \big( \hessone_{{\Gamma}} \psi
+ \bn\otimes ({\bf H} \nab_\Gamma \psi)- ({\bf H}\nab_\Gamma \psi) \otimes \bn\big)\\
&\qquad
+\int_S \nab_\Gamma \phi^\intercal ( {\rm tr}({\bf H}) {\bf H}-{\bf H}^2) \nab_\Gamma \psi
+\int_{\p S} (\nab_\Gamma \Delta_\Gamma \phi\cdot \bmu_S)\psi-\int_{\p S} \bmu^\intercal_S \underline{D}_{{\Gamma}}^2 \phi \nab_\Gamma \psi \\
&= \int_S  \hessone_{{\Gamma}} \phi:  \hessone_{{\Gamma}} \psi
+ \int_S \bn^\intercal \hessone_{{\Gamma}} \phi {\bf H} \nab_\Gamma \psi\\
&\qquad
+\int_S \nab_\Gamma \phi^\intercal ( {\rm tr}({\bf H}) {\bf H}-{\bf H}^2) \nab_\Gamma \psi
+\int_{\p S} (\nab_\Gamma \Delta_\Gamma \phi\cdot \bmu_S)\psi-\int_{\p S} \bmu^\intercal_S \underline{D}_{{\Gamma}}^2 \phi \nab_\Gamma \psi \\
&= \int_S  \hessone_{{\Gamma}} \phi:  \hessone_{{\Gamma}} \psi
+\int_S \nab_\Gamma \phi^\intercal ( {\rm tr}({\bf H}) {\bf H}-2{\bf H}^2) \nab_\Gamma \psi\\
&\qquad+\int_{\p S} (\nab_\Gamma \Delta_\Gamma \phi\cdot \bmu_S)\psi-\int_{\p S} \bmu^\intercal_S \underline{D}_{{\Gamma}}^2 \phi \nab_\Gamma \psi.
\end{align*}
Finally we calculate
\begin{align*}
\hessone_{{\Gamma}} \phi:\hessone_{{\Gamma}} \psi
- \hesstwo \phi:\hesstwo \psi
& = \hessone_{{\Gamma}} \phi:\hessone_{{\Gamma}} \psi - ({\bf P}\hessone_{{\Gamma}} \phi):({\bf P}\hessone_{{\Gamma}} \psi)\\
%
%
& = (({\bf I}-{\bf P})\hessone_{{\Gamma}} \phi):\underline{D}_{{\Gamma}}^2 \psi\\
& = (\bn^\intercal \underline{D}_{{\Gamma}}^2 \phi)\cdot (\bn^\intercal \underline{D}_{{\Gamma}}^2 \psi) = \nab_\Gamma \phi^\intercal {\bf H}^2 \nab_\Gamma \psi.
\end{align*}
and use the identity ${\rm tr}({\bf H}){\bf H} - {\bf H}^2 = K{\bf P}$ \cite{JankuhnEtal18}, to obtain the desired result.

\section{Proof of Proposition \ref{prop:NormEquivalence}}\label{app:NormE}

The proof of \eqref{ineq.1} relies on a discrete Korn-type
inequality for surface BDM spaces, which we now
define. 

Recall $\bar \calT_h$ is the set of faces of the polyhedral approximation
to $\Gamma$, and $\calT_h = \{\bp_{k_g}(\bar T):\ \bar T\in \bar \calT_h\}$.
Let $\hat T\subset \bbR^2$ be the reference simplex
with vertices $(0,0),(1,0),(0,1)$, and for $\bar T\in \bar \calT_h$,
let $F_{\bar T}:\hat T\to \bar T$ be an affine diffeomorphism.
We then set $F_T = \bp_{k_g}\circ F_{\bar T}:\hat T\to T\in \calT_h$
where $T = \bp_{k_g}(\bar T)$.  We also set $F_{T^\ell} = \bp \circ F_{T}:\hat T\to T^\ell \in \calT_h^\ell$
where $T^\ell = \bp(T)$.

For an integer $k\ge 1$, let $\bbP_k(\hat T)$ denote the space of polynomials
of degree $\le k$ on the reference triangle.
The BDM spaces defined on 
the polynomial-mapped surface mesh and the exact surface mesh are given by, respectively,
\begin{align}\label{eqn:SigmaIsoDef}
\Sigma_h &= \{\bq\in \bH({\rm div}_{\Gamma_h};\Gamma_h):\ \bq|_T = \calP_{F_T} \hat \bq,\ \exists \hat \bq \in [\bbP_{k-1}(\hat T)]^2\ \forall T\in \calT_h\},\\
\Sigma_h^\ell & = 
\{\bq\in \bH({\rm div}_{\Gamma};\Gamma):\ \bq|_T = \calP_{F^\ell_T} \hat \bq,\ \exists \hat \bq \in [\bbP_{k-1}(\hat T)]^2\ \forall T\in \calT_h\},
\end{align}
where
\[
\calP_{F_T} \hat \bq = \frac{DF_T}{\sqrt{\det(DF_{T}^\intercal DF_{T})}} \hat \bq,  \quad \calP_{F_{T^\ell}} \hat \bq = \frac{DF_{T^\ell}}{\sqrt{\det(DF_{T^\ell}^\intercal DF_{T^\ell})}} \hat \bq
\]
are the Piola transform of $\hat \bq$ with respect to $F_T$ and $F_{T^\ell}$, respectively.


\subsection{Relationships between the surface Langrage and BDM spaces}

We write $x=F_T(\hat x)$ with $x\in T$ and $\hat x\in \hat T$, 
and to ease the presentation, we set 
\[
A = A(\hat x) = DF_T(\hat x).
\]

\begin{lemma}\label{lem:CurlBDM}
    There holds ${\bf curl}_{\Gamma_h} \psi\in \Sigma_h$ for all $\psi\in V_h$
    and ${\bf curl}_{\Gamma} \psi\in \Sigma^\ell_h$ for all $\psi\in V_h^\ell$.
\end{lemma}
\begin{proof}
We only prove the first assertion, as the second is proved identically.

We first show that ${\bf curl}_{\Gamma_h} \psi|_T = \calP_{F_T} \hat \bq$ for some $\hat \bq\in [\bbP_{k-1}(\hat T)]^2\ \forall T\in \calT_h$.
Fixing $T\in \calT_h$, we calculate using the chain and product rules,
\begin{equation}\label{eqn:InvJacobian}
\Grad \hat x = (A^\intercal A)^{-1}A^\intercal =:C
\end{equation}

Next, let $\psi\in V_h$ and let $\hat \psi\in \bbP_k(\hat T)$ be related via
$\psi|_T(x) = \hat \psi(\hat x)$ with $x = F_T(\hat x)$.  
The outward normal of $\Gamma_h$ restricted to $T$
is $\bn_T = \bt_2\times \bt_1/|\bt_2\times \bt_1|$,
where we assume that $F_T$ 
is the bijection such that 
$\bt_j$ is the $j$th column of $A$. We also set 
\[
S = \begin{pmatrix}
    0 & 1\\
    -1 & 0
\end{pmatrix}.
\]
We then have by the chain rule and \eqref{eqn:InvJacobian},
\begin{align*}
({\bf curl}_{\Gamma_h} \psi)(x) = \bn_T^\times \nab \psi(x) = \bn_T^\times C^\intercal \hat \nab \hat \psi(\hat x) = 
\bn_T^\times C^\intercal S^{-1} \widehat{\bf curl}\hat \psi(\hat x),
\end{align*}
where $\widehat{\bf curl}\hat \psi(\hat x) = (\p \hat \psi/\p \hat x_2,-\p \hat \psi/\p \hat x_1)^\intercal$.
A short calculation reveals $\bn_T^\times C^\intercal S^{-1} = DF_T/|\bt_2\times \bt_1| = DF_T/\sqrt{\det(DF_T^\intercal DF_T)}$,
and so
\begin{align*}
({\bf curl}_{\Gamma_h} \psi)(x) 
& = \left(\frac{DF_{{T}}}{\sqrt{\det(DF_T^\intercal DF_T)}}\widehat{\bf curl}\hat \psi\right)(\hat x) = \calP_{F_T} \widehat {\bf curl} \hat \psi.
\end{align*}
Thus, ${\bf curl}_{\Gamma_h} \psi|_T = \calP_{F_T} \hat \bq$ for some $\hat \bq\in [\bbP_{k-1}]^2$.

It remains to show ${\bf curl}_{\Gamma_h} \psi \in \bH({\rm div}_h;\Gamma_h)$,
i.e.,  ${\bf curl}_{\Gamma_h} \psi$ has co-normal continuity.
To this end, let $e = \p T_+\cap \p T_- \in \calE_h$ be an edge with $T_{\pm}\in \calT_h$.
Let $\bmu_{\pm}$ and $\bn_{\pm}$ be the outward unit co-normal and normal, respectively, of $\p T_{\pm}$ restricted to $e$,
and set $\psi_{\pm} = \psi|_{T_{\pm}}$. We also let $\bt_{\pm} = \bmu_{\pm} \times \bn_{\pm}$
to be the vector tangent to $e$ and note that $\bt_+ = -\bt_-$.
On $e$ we have
\begin{align*}
({\bf curl}_{\Gamma_h} \psi_+)\cdot \bmu_+ + ({\bf curl}_{\Gamma_h} \psi_-)\cdot \bmu_-
& = 
(\bn_+ \times \nab_{\Gamma_h} \psi_+) \cdot \bmu_+ + (\bn_-\times \nab _{\Gamma_h} \psi_-)\cdot \bmu_-\\
& = (\bmu_+ \times \bn_+  ) \cdot \nab_{\Gamma_h} \psi_+ + (\bmu_- \times \bn_-  ) \cdot \nab_{\Gamma_h} \psi_-\\
& = \bt_+\cdot \nab_{\Gamma_h} \psi_++\bt_-\cdot \nab_{\Gamma_h} \psi_- = 0,
\end{align*}
where we used the continuity of $\psi$ in the last equality.
Thus we conclude ${\bf curl}_{\Gamma_h} \psi \in \bH({\rm div}_{\Gamma_h};\Gamma_h)$,
and so ${\bf curl}_{\Gamma_h} \psi\in \Sigma_h$.
\end{proof}

\subsection{Proof of inequalities \eqref{ineq.1}}

\begin{proof}
We start with the discrete Korn inequality in \cite{Kilicer}:
\begin{equation}
 \|\Grad_{\Gamma_h} \bq\|_{L^2(\calT_h)}
 \lesssim 
\|E_{\Gamma_{h}}\bq \|_{L^{2}(\calT_h)}
+ \left(\sum_{e \in \calE_{h}} h^{-1}\left\|[\bq ]\right\|_{L^{2}(e)}^{2}\right)^{\frac{1}{2}} + \left\|\bq\right\|_{L^{2}\left(\Gamma_{h}\right)}\quad \forall \bq\in \Sigma_h,
\end{equation}
where, on $e = \p T_+\cap \p T_-\in \calE_h$, $[\bq] = \bq_+-\bq_-$.

Using Lemma \ref{lem:CurlBDM}, we take 
$\bq = \bcurl_{\Gamma_h} \psi \in \Sigma_h$ to obtain
\begin{equation}\label{eqn:PsiLineoo}
\begin{split}
 \|\Grad_{\Gamma_h} \bcurl_{\Gamma_h} \psi \|_{L^2(\calT_h)}
& \lesssim \| H_{\Gamma_h}(\psi)\|_{L^2(\calT_h)}
+ \left(\sum_{e \in \calE_{h}}h^{-1}\left\|[\bcurl_{\Gamma_{h}}\psi]\right\|_{L^{2}(e)}^{2}\right)^{\frac{1}{2}} 
+ \|\nab_{\Gamma_{h}}\psi\|_{L^{2}(\Gamma_{h})}.
\end{split}
\end{equation}
On  $e = \p T_+\cap \p T_-$, we use
the fact that $\{\bt_{T_{\pm}}, \bmu_{T_\pm}, \bn_{T_\pm}\}$ is an orthonormal basis to write
\begin{align*}
\bcurl_{\Gamma_{h}}\psi_{\pm} 
& = (\bmu_{T_{\pm}} \cdot \nab_{\Gamma_{h}}\psi_{\pm})\bt_{T_{\pm}} - \left(\bt_{T_{\pm}} \cdot \nab_{\Gamma_{h}}\psi_{\pm}\right)\bmu_{T_{\pm}}.
\end{align*}
Thus, using the continuity of $\psi$, we have
\begin{align*}
\left[\bcurl_{\Gamma_{h}}\psi\right] 
& = \left(\left(\bmu_{T_+} \cdot \nab_{\Gamma_{h}}\psi_{+}\right)\bt_{{T}+} - \left(\bt_{{T}+} \cdot \nab_{\Gamma_{h}}\psi_{+}\right)\bmu_{{T}+}\right) - \left(\left(\bmu_{{T}-} \cdot \nab_{\Gamma_{h}}\psi_{-}\right)\bt_{{T}-} - \left(\bt_{{T}-} \cdot \nab_{\Gamma_{h}}\psi_{-}\right)\bmu_{{T}-}\right) \\
& = \jump{\nab_{\Gamma_{h}}\psi}\bt_{T_+} - \left(\bt_{T_+} \cdot \nab_{\Gamma_{h}}\psi_{+}\right)\left(\bmu_{T_+} + \bmu_{T_-}\right).
\end{align*}
%
%
It then follows from Lemma \ref{lem:GeomEstimates1} and Proposition \ref{prop:Trace} that

\begin{align*}
\sum_{e \in \calE_{h}}\left\|\left[\bcurl_{\Gamma_{h}}\psi\right]\right\|_{L^{2}(e)}^{2} 
& \lesssim \sum_{e \in \calE_{h}}\left\|\jump{\nab_{\Gamma_{h}}\psi}\right\|_{L^{2}(e)}^{2} 
+\left\|\nab_{\Gamma_{h}}\psi\right\|_{L^{2}\left(\Gamma_{h}\right)}^{2}.
\end{align*}
We apply this estimate to \eqref{eqn:PsiLineoo}, obtaining
\begin{align*}
\|\Grad_{\Gamma_{h}}\bcurl_{\Gamma_{h}} \psi\|_{L^2(\calT_h)}
& \lesssim \|H_{\Gamma_{h}}(\psi)\|_{L^{2}(\calT_h)}
+ h^{- \frac{1}{2}}\left(\sum_{e \in \calE_{h}}h_e^{-1} \|\jump{\nab_{\Gamma_{h}}\psi}\|_{L^{2}(e)}^{2}\right)^{1/2} + 
\|\nab_{\Gamma_{h}}\psi\|_{L^{2}(\Gamma_{h})} \lesssim \|\psi\|_{2,h}.
\end{align*}
Finally, we use \eqref{eqn:streamGrad} and the identity
$\bn_h^{\times}\bn_h^{\times} = - {\bf P}_h$ to get
\begin{align*}
\left|\Grad_{\Gamma_h}\bcurl_{\Gamma_h}\psi\right|^{2} & = \text{tr}\left(\left(\bn^{\times}_h\nab_{\Gamma_h}^{2}\psi\right)^{\intercal}\left(\bn^{\times}_h\nab_{\Gamma_h}^{2}\psi\right)\right) 
 = \text{tr}\left(- \nab_{\Gamma_h}^{2}\psi\bn_h^{\times}\bn_h^{\times}\nab_{\Gamma_h}^{2}\psi\right) 
 = \left|\nab_{\Gamma_h}^{2}\psi\right|^{2},
\end{align*}
and so $\|\nab_{\Gamma_h}^2 \psi\|_{L^2(\calT_h)}\lesssim \|\psi\|_{2,h}$.
Thus, the first inequality in \eqref{ineq.1} holds. 
The proof of the second inequality is nearly identical, so is omitted.
\end{proof}

\section{Proofs for Results in Section \ref{sec-Geo}}\label{app:DifferRelationsProof}
The proofs of Lemmas \ref{lem:SourceFormCons}--\ref{lem:DifferRelationsAgain} require an intermediate result, which provides relationships between differential operators on $\Gamma$ and $\Gamma_h$. This result is essentially found in \cite[Appendix C]{LarssonLarson17} and follows from the chain rule. For completeness we provide the proof.

\begin{lemma} \label{lem:ChainRuleFun}
There holds on $\Gamma_h$,
\begin{align}
\nab_{\Gamma_h}\psi & = {\bf P}_h({\bf P} - d{\bf H})(\nab_\Gamma\psi^\ell)^e, \label{eqn:gradChain} \\
\begin{split}
{\bf curl}_{\Gamma_h}\psi & = \bn_h^\times{\bf P}_h({\bf P} - d{\bf H})(\nab_\Gamma\psi^\ell)^e \\
& = \mu_h\bigg({\bf I} - \frac{\bn \otimes \bn_h}{\bn \cdot \bn_h}\bigg)({\bf I} - d{\bf H})^{- 1}(\bcurl_{\Gamma}\psi^\ell)^e,
\end{split} \label{eqn:curlChain} \\
\begin{split}
\nab_{\Gamma_h}^2\psi & = {\bf P}_h\big((\nab_\Gamma^2\psi^\ell)^e - d{\bf H}(\nab^2_\Gamma\psi^\ell)^e - d(\nab^2_\Gamma\psi^\ell)^e{\bf H} + d^2{\bf H}(\nab^2_\Gamma\psi^\ell)^e{\bf H}\big){\bf P}_h \\
& \qquad - ({\bf P}_h\bn) \otimes ({\bf H}(\nab_\Gamma\psi^\ell)^e){\bf P}_h - ({\bf P}_h{\bf H}(\nab_\Gamma\psi^\ell)^e) \otimes ({\bf P}_h\bn) \\
& \qquad - d{\bf P}_h\nab{\bf H}(\nab_\Gamma\psi^\ell)^e{\bf P}_h,
\end{split} \label{eqn:hessChain}
\end{align}
where
\begin{equation}\label{eqn:GradOfTensor}
(\nab{\bf H}\bv)_{i, j} = \sum_{k = 1}^3\frac{\p H_{i, k}}{\p x_j}v_k.
\end{equation}
\end{lemma}


\begin{proof}
\textit{Identity} \eqref{eqn:gradChain}: Since $\psi^\ell$ is constant along normals, we have $\psi^\ell = \psi^\ell \circ \bp$. Likewise, we have $\bn = \bn\circ \bp$. Thus, taking the gradient of $\psi^\ell = \psi^\ell \circ \bp$ yields
\begin{align}
\begin{split}
\nab\psi^\ell & = ({\bf I} - d{\bf H}){\bf P}\nab\psi^\ell \circ \bp \\
& = ({\bf I} - d{\bf H})({\bf P} \circ \bp)(\nab\psi^\ell \circ \bp) \\
& = ({\bf I} - d{\bf H})(\nab_{\Gamma}\psi^\ell)^e.
\end{split} \label{gradpsiell}
\end{align}
Consequently,
\begin{align*}
\nab_{\Gamma_h}\psi = {\bf P}_h\nab\psi^\ell = {\bf P}_h({\bf I} - d{\bf H})(\nab_{\Gamma}\psi^\ell)^e,
\end{align*}
or
\begin{align*}
(\nab_{\Gamma_h}\psi)^\ell = {\bf P}_h({\bf I} - d{\bf H})\nab_{\Gamma}\psi^\ell.
\end{align*}

\textit{Identities} \eqref{eqn:curlChain}: The first identity in \eqref{eqn:curlChain} follows from \eqref{eqn:gradChain} and the definition of the surface curl operator. To prove the second identity in \eqref{eqn:curlChain}, we recall
\begin{align*}
{\bf curl}_{\Gamma_h}\psi = \bn_h^\times{\bf P}_h({\bf P} - d{\bf H})(\nab_\Gamma\psi^\ell)^e \quad \Longrightarrow \quad ({\bf curl}_{\Gamma_h}\psi)^\ell = \bn_h^\times{\bf P}_h({\bf P} - d{\bf H})\nab_\Gamma\psi^\ell.
\end{align*}
Note that
\[
- \bn^\times\bcurl_{\Gamma}\psi^\ell = - \bn^\times\bn^\times\nab_\Gamma\psi^\ell = {\bf P}\nab_\Gamma\psi^\ell = \nab_\Gamma\psi^\ell.
\]
Thus,
\begin{align*}
({\bf curl}_{\Gamma_h}\psi)^\ell & = - \bn_h^\times{\bf P}_h({\bf P} - d{\bf H})\bn^\times\bcurl_\Gamma \psi^\ell \\
& = \big(- \bn_h^\times\bn^\times + d\bn_h^\times{\bf H}\bn^\times\big)\bcurl_\Gamma\psi^\ell.
\end{align*}
Now a direct calculation shows
\[
{\bf H}\bn^\times - ({\bf H}\bn^\times)^\intercal = (\nab \cdot \bn)\bn^\times = {\rm tr}({\bf H})\bn^\times.
\]
Consequently by the skew-symmetry of $\bn^\times$, there holds
\begin{align*}
\bn_h^\times{\bf H}\bn^\times 
%
%
& = \bn_h^\times\bn^\times({\rm tr}({\bf H}){\bf I} - {\bf H}).
\end{align*}
Thus, we have
\begin{align*}
({\bf curl}_{\Gamma_h}\psi)^\ell 
& = \bn_h^\times\bn^\times\big(\big[d{\rm tr}({\bf H}) - 1\big]{\bf I} - d{\bf H}\big)\bcurl_\Gamma\psi^\ell \\
& = ((\bn \cdot \bn_h){\bf I} - \bn \otimes \bn_h)
\big(\big[1 - d{\rm tr}({\bf H})\big]{\bf I} + d{\bf H}\big)\bcurl_\Gamma\psi^\ell.
%
\end{align*}
Noting that
\begin{align*}
({\bf I} - d{\bf H})\big(\big[1 - d{\rm tr}({\bf H})\big]{\bf I} + d{\bf H}\big)
& = {\bf I} - d{\rm tr}({\bf H}){\bf I} + d^2({\rm tr}({\bf H}){{\bf H}} - {\bf H}^2) \\
& = {\bf I} - d{\rm tr}({\bf H}){\bf I} + d^2K{\bf P},
\end{align*}
and $1 - d{\rm tr}({\bf H}) + d^2K = (1 - d\kappa_1)(1 - d\kappa_2)$, we have
\begin{align*}
({\bf curl}_{\Gamma_h}\psi)^\ell 
& = (1 - d\kappa_1)(1 - d\kappa_2)((\bn \cdot \bn_h){\bf I} - \bn \otimes \bn_h)({\bf I} - d{\bf H})^{- 1}\bcurl_\Gamma\psi^\ell \\
%
%
& = \mu_h\bigg({\bf I} - \frac{\bn \otimes \bn_h}{\bn \cdot \bn_h}\bigg)({\bf I} - d{\bf H})^{- 1}\bcurl_\Gamma\psi^\ell.
\end{align*}

\textit{Identity} \eqref{eqn:hessChain}: We take the derivative of
\begin{align*}
\frac{\p \psi^\ell}{\p x_i}(x) = \big(\nab\bp(x)\nab\psi^\ell(\bp(x))\big)_i = \big(\nab\bp(x)^\intercal\nab\psi^\ell(\bp(x))\big)_i = \sum_{k = 1}^3\frac{\p p_k}{\p x_i}(x)\frac{\p \psi^\ell}{\p x_k}(\bp(x)),
\end{align*}
and use the chain rule to compute
\begin{align} \label{eqn:HessChain1}
(\nab^2\psi^\ell)_{i, j} = \big(({\bf I} - d{\bf H}){\bf P}(\nab^2\psi^\ell)^e({\bf I} - d{\bf H}){\bf P}\big)_{i, j} + \sum_{k = 1}^3(\nab\psi^\ell)_k^e\frac{\p^2p_k}{\p x_i\p x_j}.
\end{align}
We have
\begin{align*}
\frac{\p^2 p_k}{\p x_i\p x_j} & = - \frac{\p n_k}{\p x_j}n_i - n_k\frac{\p n_j}{\p x_i} - n_j\frac{\p n_k}{\p x_i} - d\frac{\p^2 n_k}{\p x_i\p x_j} \\
& = - H_{j, k}n_i - n_kH_{i, j} - n_jH_{i, k} - d\frac{\p H_{i, k}}{\p x_j},
\end{align*}
and so
\begin{align}
\begin{split} \label{eqn:HessChain2}
& \sum_{k = 1}^3(\nab\psi^\ell)^e_k\frac{\p^2 p_k}{\p x_i\p x_j} \\
& = - \big(\bn \otimes ({\bf H}(\nab\psi^\ell)^e) + {\bf H}(\bn \cdot (\nab\psi^\ell)^e) + ({\bf H}(\nab\psi^\ell)^e) \otimes \bn + d\nab{\bf H}(\nab\psi^\ell)^e\big)_{i, j} \\
& = - \big(\bn \otimes ({\bf H}(\nab_\Gamma\psi^\ell)^e) + ({\bf H}(\nab_\Gamma\psi^\ell)^e) \otimes \bn + d\nab{\bf H}(\nab_{\Gamma}\psi^\ell)^e\big)_{i, j},
\end{split}
\end{align}
where we used the fact that $\psi^\ell$ is constant along normals in the last equality.
Combining \eqref{eqn:HessChain1}--\eqref{eqn:HessChain2} and the identity ${\bf H}{\bf P} = {\bf P}{\bf H}$ yields
\begin{align*}
\nab^2\psi^\ell & = ({\bf I} - d{\bf H}){\bf P}(\nab^2\psi^\ell)^e{{\bf P}}({\bf I} - d{\bf H}) \\
& \qquad - \big(\bn \otimes ({\bf H}(\nab_\Gamma\psi^\ell)^e) + ({\bf H}(\nab_\Gamma\psi^\ell)^e) \otimes \bn + d\nab{\bf H}(\nab_\Gamma\psi^\ell)^e\big).
\end{align*}
Recalling $\nab_\Gamma^2\psi^\ell = {\bf P}\nab^2\psi^\ell{\bf P}$, we have
\begin{align*}
\nab^2\psi^\ell
& = (\nab_\Gamma^2\psi^\ell)^e - d{\bf H}(\nab^2_\Gamma\psi^\ell)^e - d(\nab^2_{\Gamma}\psi^\ell)^e{\bf H} + d^2{\bf H}(\nab^2_\Gamma\psi^\ell)^e{\bf H} \\
& \qquad - \big(\bn \otimes ({\bf H}(\nab_\Gamma\psi^\ell)^e) + ({\bf H}(\nab_\Gamma\psi^\ell)^e) \otimes \bn + d\nab{\bf H}(\nab_\Gamma\psi^\ell)^e\big),
\end{align*}
and therefore,
\begin{align*}
 \nab_{\Gamma_h}^2\psi
%
& = {\bf P}_h(\nab_\Gamma^2\psi^\ell)^e{\bf P}_h - d{\bf P}_h{\bf H}(\nab^2\psi^\ell_\Gamma)^e{\bf P}_h - d{\bf P}_h(\nab^2_\Gamma\psi^\ell)^e{\bf H}{\bf P}_h + d^2{\bf P}_h{\bf H}(\nab^2_\Gamma\psi^\ell)^e{\bf H}{\bf P}_h \\
& \qquad - \big(({\bf P}_h\bn) \otimes ({\bf H}(\nab_\Gamma\psi^\ell)^e){\bf P}_h + ({\bf P}_h{\bf H}(\nab_\Gamma\psi^\ell)^e) \otimes ({\bf P}_h\bn) + d{\bf P}_h\nab{\bf H}(\nab_\Gamma\psi^\ell)^e{\bf P}_h\big).
\end{align*}
\end{proof}

To ease notation in the rest of this section, we set
\begin{align}
{\bf M} := ({\bf I} - d{\bf H})^{- 1}, \qquad {\bf N} := {\bf I} - \frac{\bn \otimes \bn_h}{\bn \cdot \bn_h}, \qquad {\bf L}:=\mu_h^{-1}({\bf P}-d {\bf H}),
\label{notations.MN}
\end{align}
so that the second identity in Lemma \ref{lem:ChainRuleFun} reads 
\begin{align*}
\bcurl_{\Gamma_h}\psi = \mu_h {\bf N}{\bf M} (\bcurl_\Gamma \psi^\ell)^e\ \ (\psi\in H^1(\Gamma_h)) \ \Longrightarrow \  {\bf curl}_{\Gamma_h}\psi^e = \mu_h {\bf N}{\bf M} (\bcurl_\Gamma\psi)^e\ \ (\psi\in H^1(\Gamma)),
\end{align*}
and we also have
\begin{align*}
(\bcurl_\Gamma\psi^\ell)^e = {\bf L}({\bf curl}_{\Gamma_h} \psi).
\end{align*}
We also note that
\begin{equation} \label{eqn:LMBound}
\big|{\bf P} - {\bf L}\big| \lesssim h^{k_g + 1}, \quad \text{and} \quad \big|{\bf I} - {\bf M}\big| \lesssim h^{k_g + 1}.
\end{equation}

\subsection{Proof of Lemma \ref{lem:SourceFormCons}}
\begin{proof}
We make a change of variables   to obtain
\begin{align*}
\ell_h(\psi) - \ell(\psi^\ell) 
& = \int_{\Gamma}\mu_h^{- 1}{\bm f}^\ell_h \cdot (\bcurl_{\Gamma_h}\psi)^\ell - \int_{\Gamma}{\bm f} \cdot \bcurl_{\Gamma}\psi^\ell \\
\end{align*}
Applying \eqref{eqn:curlChain}, we have
\begin{align*}
\int_{\Gamma}\mu_h^{-1} {\bm f}_h^\ell \cdot (\bcurl_{\Gamma_h} \psi)^\ell
& = \int_\Gamma {\bm f}_h^\ell \cdot ({\bf N}{\bf M} \bcurl_\Gamma \psi^\ell)\\
%
%
 & = \int_\Gamma (\bcurl_\Gamma \psi^\ell)^\intercal {\bf M} \left( {\bf I}- \frac{\bn_h \otimes \bn}{\bn \cdot \bn_h}\right) 
 {\bm f}_h^\ell\\
 %
   & = \int_\Gamma (\bcurl_\Gamma \psi^\ell)^\intercal {\bf M} \left( {\bf I}- \frac{({\bf P}\bn_h) \otimes ({\bf P}_h \bn)}{\bn \cdot \bn_h}\right) 
 {\bm f}_h^\ell\\
 &\qquad -  \int_\Gamma  (\bcurl_\Gamma \psi^\ell)^\intercal \left(({\bf M} \bn)\otimes ({\bf P}_h \bn)\right) {\bm f}_h^\ell,
\end{align*}
where we used ${\bm f}^\ell _h\cdot \bn_h = 0$  in the last  equality. Since ${\bf M}\bn = \bn$
and $(\bcurl \psi^\ell)\cdot \bn = 0$, we conclude
\begin{align*}
 \int_{\Gamma}\mu_h^{-1} {\bm f}_h^\ell \cdot (\bcurl_{\Gamma_h} \psi)^\ell
   & = \int_\Gamma (\bcurl_\Gamma \psi^\ell)^\intercal {\bf M} \left( {\bf I}- \frac{({\bf P}\bn_h) \otimes ({\bf P}_h \bn)}{\bn \cdot \bn_h}\right) 
 {\bm f}_h^\ell = \int_\Gamma {\bm F}_h^\ell \cdot \bcurl_\Gamma \psi^\ell,
\end{align*}
and so
\begin{align*}
    \ell_h(\psi) - \ell(\psi^\ell)  = \int_\Gamma ({\bm F}_h^\ell-{\bm f}) \cdot \bcurl_\Gamma \psi^\ell\le \|{\bm f}-{\bm F}_h^\ell\|_{L^2(\Gamma)}\tbar{\psi^\ell}_H.
\end{align*}

 \end{proof}

\subsection{Proof of Lemma \ref{lem:DifferRelations}}
\begin{proof}\
\begin{itemize}[leftmargin=*]
\item\textit{Estimate} \eqref{ApproxOper.1}: By \eqref{eqn:deformGrad},
\begin{align*}
(H_{\Gamma}(\psi))^{e} & = \frac{1}{2}\big(\bn^{\times}(\nab_{\Gamma}^{2}\psi)^{e} - (\nab_{\Gamma}^{2}\psi)^{e}\bn^{\times}\big), \\
H_{\Gamma_{h}}(\psi^{e}) & = \frac{1}{2}\big(\bn_{h}^{\times}\nab_{\Gamma_{h}}^{2}\psi^{e} - \nab_{\Gamma_{h}}^{2}\psi^{e}\bn_{h}^{\times}\big).
\end{align*}
Then, the triangle inequality, \eqref{eqn:GeoBounds} and a change of integration domain 
(cf.~\eqref{eqn:muCloseToOne}) yield
\begin{align}
\begin{split}
& \|(H_{\Gamma}(\psi))^{e} - H_{\Gamma_{h}}(\psi^{e})\|_{L^{2}(\calT_h)} \\
& \leq \frac{1}{2}\|\bn^{\times}(\nab_{\Gamma}^{2}\psi)^{e} - \bn_{h}^{\times}\nab_{\Gamma_{h}}^{2}\psi^{e}\|_{L^{2}(\calT_h)} 
+ \frac{1}{2}\|(\nab_{\Gamma}^{2}\psi)^{e}\bn^{\times} - \nab_{\Gamma_{h}}^{2}\psi^{e}\bn_{h}^{\times}\|_{L^{2}(\calT_h)} \\
& \leq \frac{1}{2}\|\big(\bn^{\times} - \bn_{h}^{\times}\big)(\nab_{\Gamma}^{2}\psi)^{e}\|_{L^{2}(\calT_h)} 
+ \frac{1}{2}\|\bn_{h}^{\times}\big((\nab_{\Gamma}^{2}\psi)^{e} - \nab_{\Gamma_{h}}^{2}\psi^{e}\big)\|_{L^{2}(\calT_h)} \\
& \qquad + \frac{1}{2}\|(\nab_{\Gamma}^{2}\psi)^{e}\big(\bn^{\times} - \bn_{h}^{\times}\big)\|_{L^{2}(\calT_h)} 
+ \frac{1}{2}\|\big((\nab_{\Gamma}^{2}\psi)^{e} - \nab_{\Gamma_{h}}^{2}\psi^{e}\big)\bn_{h}^{\times}\|_{L^{2}(\calT_h)} \\
& \lesssim h^{k_{g}}\|(\nab_{\Gamma}^{2}\psi)^{e}\|_{L^{2}(\calT_h)} 
+ \|(\nab_{\Gamma}^{2}\psi)^{e} - \nab_{\Gamma_{h}}^{2}\psi^{e}\|_{L^{2}(\calT_h)} \\
& \lesssim h^{k_{g}}\|\nab_{\Gamma}^{2}\psi\|_{L^{2}(\calT_h^\ell)} 
+ \|\nab_{\Gamma}^{2}\psi - (\nab_{\Gamma_{h}}^{2}\psi^{e})^{\ell}\|_{L^{2}(\calT_h^\ell)}.
\end{split} \label{ApproxOper.7}
\end{align}
Note that from \eqref{eqn:hessChain},
\begin{align*}
& (\nab_{\Gamma_{h}}^{2}\psi^{e})^{\ell} - \nab_{\Gamma}^{2}\psi \\
& = {\bf P}_{h}\nab_{\Gamma}^{2}\psi{\bf P}_{h} - \nab_{\Gamma}^{2}\psi - {\bf P}_{h}\big(d{\bf H}\nab_{\Gamma}^{2}\psi + d\nab_{\Gamma}^{2}\psi{\bf H} - d^{2}{\bf H}\nab_{\Gamma}^{2}\psi{\bf H}\big){\bf P}_{h} \\
& \qquad - ({\bf P}_{h}\bn) \otimes ({\bf H}\nab_{\Gamma}\psi){\bf P}_{h} - ({\bf P}_{h}{\bf H}\nab_{\Gamma}\psi) \otimes ({\bf P}_{h}\bn) - d{\bf P}_{h}\nab{\bf H}\nab_{\Gamma}\psi{\bf P}_{h} \\
& = \nab_{\Gamma}^{2}\psi(\bn - \bn_{h}) \otimes \bn_{h} + \bn_{h} \otimes (\bn - \bn_{h})\nab_{\Gamma}^{2}\psi + \bn_{h} \otimes (\bn - \bn_{h})\nab_{\Gamma}^{2}\psi(\bn - \bn_{h}) \otimes \bn_{h} \\
& \qquad - {\bf P}_{h}\big(d{\bf H}\nab_{\Gamma}^{2}\psi + d\nab_{\Gamma}^{2}\psi{\bf H} - d^{2}{\bf H}\nab_{\Gamma}^{2}\psi{\bf H}\big){\bf P}_{h} \\
& \qquad - ({\bf P}_{h}\bn) \otimes ({\bf H}\nab_{\Gamma}\psi){\bf P}_{h} - ({\bf P}_{h}{\bf H}\nab_{\Gamma}\psi) \otimes ({\bf P}_{h}\bn) - d{\bf P}_{h}\nab{\bf H}\nab_{\Gamma}\psi{\bf P}_{h}.
\end{align*}
Thus, \eqref{eqn:GeoBounds} and \eqref{GeoApprox.1} imply
\begin{align}
\|\nab_{\Gamma}^{2}\psi - (\nab_{\Gamma_{h}}^{2}\psi^{e})^{\ell}\|_{L^{2}(\calT_h^\ell)} 
\lesssim h^{k_{g}}\|\nab_{\Gamma}\psi\|_{L^{2}(\calT_h^\ell)} + h^{k_{g}}\|\nab_{\Gamma}^{2}\psi\|_{L^{2}(\calT_h^\ell)} 
\lesssim h^{k_g}\tbar{\psi}_H \label{ApproxOper.8},
\end{align}
and \eqref{ApproxOper.1} follows from \eqref{ApproxOper.7} and \eqref{ApproxOper.8}. \\

\item \textit{Estimate} \eqref{ApproxOper.2}:
By \eqref{eqn:gradChain} and ${\bf P}\nab_\Gamma \psi = \nab_\Gamma \psi$ we have
\begin{align*}
(\nab_{\Gamma_h}\psi^e)^\ell & = {\bf P}_h({\bf P} - d{\bf H})\nab_\Gamma \psi \\
& = \nab_\Gamma \psi + ({\bf P}_h - {\bf P})\nab_\Gamma \psi - d{\bf P}_h{\bf H}\nab_\Gamma \psi.
\end{align*}
Hence, due to a change of integration domain, \eqref{eqn:GeoBounds} and the definition of $\tbar{\cdot}_H$,
\[
\|(\nab_{\Gamma}\psi)^{e} - \nab_{\Gamma_{h}}\psi^{e}\|_{L^{2}(\Gamma_{h})} \lesssim \|\nab_{\Gamma}\psi - (\nab_{\Gamma_{h}}\psi^{e})^{\ell}\|_{L^{2}(\Gamma)} \lesssim h^{k_{g}}\|\nab_{\Gamma}\psi\|_{L^{2}(\Gamma)} \lesssim h^{k_{g}}\tbar{\psi}_H.
\]

\item \textit{Estimate} \eqref{ApproxOper.3}: 
For an edge $e\in \calE_h$, we write $e = \p T_+\cap \p T_-$ with $T_{\pm}\in \calT_h$.
Recalling $\nab_{\Gamma_h} \psi^e = {\bf P}_h ({\bf P}-d {\bf H}) (\nab_\Gamma \psi)^e$ 
and ${\bf P}_h \bmu_h = \bmu_h$, we have
\begin{align*}
\jump{\nab_{\Gamma_h} \psi^e} 
&= ({\bf P}-d {\bf H})\left[\big((\nab_\Gamma \psi)^e\big)^+\cdot \bmu_{T_+} +
\big((\nab_\Gamma \psi)^e\big)^-\cdot \bmu_{T_-}\right]\\
&= ({\bf P}-d {\bf H})\left[\big((\nab_\Gamma \psi)^+\big)^e\cdot \bmu_{T_+} +
\big((\nab_\Gamma \psi)^-\big)^e\cdot \bmu_{T_-}\right]\\
&= ({\bf P}-d {\bf H})(\jump{\nab_\Gamma \psi})^e
+({\bf P}-d {\bf H})\left[\big((\nab_\Gamma \psi)^+\big)^e\cdot (\bmu_{T_+} - \bmu^e_{T^\ell_+}) +
\big((\nab_\Gamma \psi)^-\big)^e\cdot (\bmu_{T_-}-\bmu^e_{T^\ell_-})\right],
\end{align*}
where we used the continuity of $({\bf P}-d {\bf H})$. We then have
\begin{equation}\label{eqn:JumpSplit}
\begin{split}
\big\|h^{-\frac12} \big((&\jump{\nab_{\Gamma}\psi})^{e} - \jump{\nab_{\Gamma_{h}}\psi^{e}}\big)\big\|_{L^2(\calE_h)}
 \le \big\|h^{-\frac12} \big[{\bf I}-({\bf P}-d {\bf H})\big](\jump{\nab_\Gamma \psi})^e\big\|_{L^2(\calE_h)}\\
&+ \big\|h^{-\frac12}({\bf P}-d {\bf H})\big[\big((\nab_\Gamma \psi)^+\big)^e\cdot (\bmu_{T_+} - \bmu^e_{T^\ell_+}) +
\big((\nab_\Gamma \psi)^-\big)^e\cdot (\bmu_{T_-}-\bmu^e_{T^\ell_-})\big]\big\|_{L^2(\calE_h)}\\
&=:I_1+I_2.
\end{split}
\end{equation}
To bound $I_1$, we use ${\bf P}(\jump{\nab_\Gamma \psi})^e = (\jump{\nab_\Gamma \psi})^e$
and \eqref{eqn:GeoBounds} to obtain
\begin{align}\label{eqn:JI1}
    I_1\lesssim h^{k_g+1} \big\|h^{-\frac12} (\jump{\nab_\Gamma \psi})^e\big\|_{L^2(\calE_h)}\lesssim h^{k_g+1}\tbar{\psi}_H.
\end{align}
Next we write
\begin{align*}
({\bf P}-d {\bf H})\big((\nab_\Gamma \psi)^{\pm}\big)^e \cdot (\bmu_{T_{\pm}}-\bmu^e_{T^\ell_{\pm}})\big)
& = \big((\nab_\Gamma \psi)^{\pm}\big)^e \cdot ({\bf P}\bmu_{T_{\pm}}-\bmu^e_{T^\ell_{\pm}})\\
&\qquad - d {\bf H}\big((\nab_\Gamma \psi)^{\pm}\big)^e \cdot (\bmu_{T_{\pm}}-\bmu^e_{T^\ell_{\pm}}),
\end{align*}
and apply \eqref{GeoApprox.3} and \eqref{eqn:GeoBounds}:
\begin{align*}
\big\|h^{-\frac12} ({\bf P}-d {\bf H})\big((\nab_\Gamma \psi)^{\pm}\big)^e \cdot (\bmu_{T_{\pm}}-\bmu^e_{T^\ell_{\pm}})\big)\big\|_{L^2(\calE_h)}\lesssim h^{k_g+1}
\big\|h^{-\frac12}\big((\nab_\Gamma \psi)^{\pm}\big)^e\big\|_{L^2(\calE_h)}.
\end{align*}
 It then follows from a trace inequality and the definition of $\tbar{\cdot}_H$ that
 \begin{align}\label{eqn:JI2}
 I_2 \lesssim h^{k_g} \tbar{\psi}_H.
 \end{align}
 The desired estimate \eqref{ApproxOper.3} now follows from \eqref{eqn:JumpSplit}--\eqref{eqn:JI2}.

\item \textit{Estimate} \eqref{ApproxOper.4}: It is sufficient to show that
\begin{align}\label{eqn:SufficesABC}
h \sum_{T\in \calT_h} \|(\bt^{\intercal}H_{\Gamma}(\psi)\bmu)^{e} - \bt_h^{\intercal}H_{\Gamma_{h}}(\psi^{e})\bmu_h\|^2_{L^{2}(\p T)} \lesssim h^{2 k_{g}}\tbar{\psi}_H^2
\end{align}
By similar arguments as in the proof of \eqref{ApproxOper.1}, \eqref{GeoApprox.4}, and \eqref{GeoApprox.5}, we have
\begin{align*}
& \|(\bt^{\intercal}H_{\Gamma}(\psi)\bmu)^{e} - \bt_h^{\intercal}H_{\Gamma_{h}}(\psi^{e})\bmu_h\|_{L^{2}(\p T)}^2 \\
& \leq \|\big(\bt^e - \bt_h\big)^\intercal(H_{\Gamma}(\psi))^{e}\bmu^e\|^2_{L^{2}(\p T)} + \|\bt^{\intercal}_h\big((H_{\Gamma}(\psi))^{e} - H_{\Gamma_{h}}(\psi^{e})\big)\bmu^e\|^2_{L^{2}(\p T)} \\
& \qquad + \|\bt^{\intercal}_hH_{\Gamma_{h}}(\psi^{e})\big(\bmu^e - \bmu_h\big)\|^2_{L^{2}(\p T)} \\
 & \lesssim h^{2 k_{g}}\|(H_{\Gamma}(\psi))^{e}\|_{L^{2}(\p T)}^2 + \|(H_{\Gamma}(\psi))^{e} - H_{\Gamma_{h}}(\psi^{e})\|_{L^{2}(\p T)}^2 + h^{2k_{g}}\|H_{\Gamma_{h}}(\psi^{e})\|_{L^{2}(\p T)}^2 \\
 & \lesssim h^{2k_{g}}(|\psi |_{H^{2}(\p T^\ell)}^2 + |\psi|_{H^{1}(\p T^\ell)}^2).
\end{align*}
Summing over $T\in \calT_h$ and applying  Proposition \ref{prop:Trace} yields
\[
h \sum_{T\in \calT_h} \|(\bt^{\intercal}H_{\Gamma}(\psi)\bmu)^{e} - \bt^{\intercal}_hH_{\Gamma_{h}}(\psi^{e})\bmu_h\|^2_{L^{2}(\p T)}\lesssim h^{2k_g} \sum_{T^\ell\in \calT_h^\ell} \left(|\psi|_{H^2(T^\ell)}^2 + h^2 |\psi|^2_{H^3(T^\ell)}\right)\lesssim h^{2k_g} \tbar{\psi}_H^2.
\]
Thus, \eqref{eqn:SufficesABC} holds and therefore \eqref{ApproxOper.4} holds.
\end{itemize}
\end{proof}

\subsection{Proof of Lemma \ref{lem:DifferRelationsAgain}} \label{app:DFA}
\begin{proof}\
\begin{itemize}[leftmargin=*]
\item \textit{Estimate} \eqref{F.1}: To prove the estimate,  we first write the left hand side of \eqref{F.1} in terms of $\bu := \bcurl_\Gamma\psi$, $\bv := \bcurl_\Gamma\chi$ and $\bw := \bcurl_{\Gamma_h}\chi^e$. We then prove estimates in terms of $\bu$, $\bv$, $\bw$, and finally transform the result back to the right hand side of \eqref{F.1}. We divide the proof into seven steps.

{\em Step (i), rewrite the left hand side of \eqref{F.1}}: By Definition \ref{def:2.1},
\begin{align}
\begin{split}
(H_\Gamma(\psi))^e & = (E_\Gamma(\bcurl_\Gamma\psi))^e = (E_\Gamma(\bu))^e, \\
(H_\Gamma(\chi))^e & = (E_\Gamma(\bcurl_\Gamma\chi))^e = (E_\Gamma(\bv))^e, \\
H_{\Gamma_h}(\chi^e) & = E_{\Gamma_h}(\bcurl_{\Gamma_h}\chi^e) = E_{\Gamma_h}(\bw).
\end{split} \label{F.Def}
\end{align}
Due to \eqref{F.Def}, the definition of $E_\Gamma(\cdot)$ and the symmetry of $\left(E_\Gamma(\bu)\right)^e$,
\begin{align}
\begin{split}
\big((H_\Gamma(\psi))^e, (H_\Gamma(\chi))^e - H_{\Gamma_h}(\chi^e)\big)_{\calT_h} & = \big((E_\Gamma(\bu))^e, (E_\Gamma (\bv))^e - E_{\Gamma_h}(\bw)\big)_{\calT_h} \\
& = \big((E_\Gamma(\bu))^e, (\Grad_\Gamma\bv)^e - \Grad_{\Gamma_h}\bw\big)_{\calT_h}.
\end{split} \label{F.5}
\end{align}

{\em Step (ii),expand $({\Grad}_\Gamma\bv)^e$ in \eqref{F.5}}: 
Recall the matrix notations in \eqref{notations.MN}.
By \eqref{eqn:curlChain}, there holds
\begin{align*}
\bw = \mu_h{\bf N}{\bf M}\bv^e,
\end{align*}
which implies $\bv = ({\bf L}\bw) \circ \bp^{- 1}$. Consequently,
\begin{align}
\begin{split}
\Grad_\Gamma\bv & = {\bf P}\Grad\bv{\bf P} \\
& = ({\bf L} \circ \bp^{- 1})\Grad(\bw \circ \bp^{- 1}){\bf P} + {\bf P}\nab({\bf L} \circ \bp^{- 1})(\bw \circ \bp^{- 1}){\bf P},
\end{split} \label{F.7}
\end{align}
where we use ${\bf P}{\bf L} = {\bf L}$. By the chain rule (cf.~\cite[(A.1) and (A.5)]{DemlowNeilan24}), we have
\begin{align}
\Grad(\bw \circ \bp^{- 1}){\bf P} & = (\Grad\bw{\bf P}_h{\bf N}{\bf M}) \circ \bp^{- 1}, \label{F.8} \\
{\bf P}\nab({\bf L} \circ \bp^{- 1})(\bw \circ \bp^{- 1}){\bf P} & = {\bf P}(\nab{\bf L}\bw{\bf P}_h{\bf N}{\bf M}) \circ \bp^{- 1}. \label{F.9}
\end{align}
It follows from \eqref{F.7}--\eqref{F.9} that
\begin{align}
\begin{split}
(\Grad_\Gamma\bv)^e & = {\bf L}\Grad\bw{\bf P}_h{\bf N}{\bf M} + {\bf P}\nab{\bf L}\bw{\bf P}_h{\bf N}{\bf M} \\
& = ({\bf L}\Grad\bw + {\bf P}\nab{\bf L}\bw){\bf P}_h{\bf N}{\bf M}.
\end{split} \label{F.10}
\end{align}
In addition, using $\bw = {\bf P}_h\bw$,
\begin{align*}
\Grad\bw = \Grad({\bf P}_h\bw) = {\bf P}_h\Grad\bw + \nab{\bf P}_h\bw. 
\end{align*}
By setting ${\bf H}_h = \Grad\bn_h$ (where the gradient is applied piecewise with respect to $\calT_h$), a short calculation shows $\nab{\bf P}_h\bw = - \bn_h \otimes ({\bf H}_h\bw)$, and therefore
\begin{align}
\Grad\bw = {\bf P}_h\Grad\bw - \bn_h \otimes ({\bf H}_h\bw). \label{F.13}
\end{align}
Due to \eqref{F.10} and \eqref{F.13} and the identity ${\bf P}_h {\bf N} = {\bf N}$,
\begin{align}
\begin{split}
(\Grad_\Gamma\bv)^e & = \big({\bf L}{\bf P}_h\Grad\bw - ({\bf L}\bn_h) \otimes ({\bf H}_h\bw) + {\bf P}\nab{\bf L}\bw\big){\bf P}_h{\bf N}{\bf M} \\
& = \big({\bf L}\Grad_{\Gamma_h}\bw - ({\bf L}\bn_h) \otimes ({\bf H}_h\bw) + {\bf P}\nab{\bf L}\bw\big){\bf P}_h{\bf N}{\bf M} \\
& = \Grad_{\Gamma_h}\bw + \big[{\bf L} - {\bf P}_h\big]\Grad_{\Gamma_h}\bw{\bf N}{\bf M} + \Grad_{\Gamma_h}\bw\big[{\bf N}{\bf M} - {\bf P}_h\big] \\
& \qquad - ({\bf L}\bn_h) \otimes ({\bf H}_h\bw){\bf N}{\bf M} + {\bf P}\nab{\bf L}\bw{\bf P}_h{\bf N}{\bf M}.
\end{split} \label{F.14}
\end{align}
We plug \eqref{F.14} into \eqref{F.5} to obtain
\begin{align}
\begin{split}
& \big((H_\Gamma(\psi))^e, (H_\Gamma(\chi))^e - H_{\Gamma_h}(\chi^e)\big)_{\calT_h} \\
& = \big((E_{\Gamma}(\bu))^e, \big[{\bf L} - {\bf P}_h\big]\Grad_{\Gamma_h}\bw{\bf N}{\bf M}\big)_{{\calT}_h} + \big((E_{\Gamma}(\bu))^e, \Grad_{\Gamma_h}\bw\big[{\bf N}{\bf M} - {\bf P}_h\big]\big)_{\calT_h} \\
& \qquad - \big((E_{\Gamma}(\bu))^e, ({\bf L}\bn_h) \otimes ({\bf H}_h\bw){\bf N}{\bf M}\big)_{{\calT}_h} + \big((E_\Gamma(\bu))^e, {\bf P}\nab{\bf L}\bw{\bf P}_h{\bf N}{\bf M}\big)_{\calT_h} \\
& =: I_1 + I_2 + I_3 + I_4.
\end{split} \label{F.1-Estimate}
\end{align}
We now bound $I_j$ separately for each $j \in \{1, 2, 3, 4\}$.

{\em Step (iii), estimate of $I_1$}: Note that $(E_\Gamma(\bu))^e({\bf P} - {\bf I}) = 0$. Thus, by Holder's inequality, \eqref{eqn:GeoBounds} and \eqref{eqn:muCloseToOne},
\begin{align}
\begin{split}
I_1 & = \big((E_\Gamma(\bu))^e, \big[{\bf L} - {\bf P}_h\big]\Grad_{\Gamma_h}\bw{\bf N}{\bf M}\big)_{\calT_h} \\
& = \big((E_\Gamma(\bu))^e, 
(\mu_h^{- 1}({\bf P} - {\bf P}_h) - {d \mu_h^{-1 }}{\bf H} + \mu_h^{- 1}(1 - \mu_h){\bf P}_h)\Grad_{\Gamma_h}\bw{\bf N}{\bf M}\big)_{\calT_h} \\
& = \big((E_\Gamma(\bu))^e, \mu_h^{- 1}({\bf P} - {\bf I})\Grad_{\Gamma_h}\bw{\bf N}{\bf M}\big)_{\calT_h} \\
& \qquad - \big((E_\Gamma(\bu))^e, {d \mu_h^{- 1}}{\bf H}\Grad_{\Gamma_h}\bw{\bf N}{\bf M}\big)_{\calT_h} + \big((E_\Gamma(\bu))^e, \mu_h^{- 1}(1 - \mu_h)\Grad_{\Gamma_h}\bw{\bf N}{\bf M}\big)_{\calT_h} \\
& \lesssim h^{k_g + 1}\|(E_\Gamma(\bu))^e\|_{L^2(\Gamma_h)}\|\bw\|_{H^1(\Gamma_h)}.
\end{split} \label{F.1-Estimate.I}
\end{align}

{\em Step (iv), estimate of $I_2$}: Writing ${\bf N} - {\bf P}_h = - (\bn \cdot \bn_h)^{- 1}({\bf P}_h \bn) \otimes \bn_h$, and applying \eqref{eqn:LMBound} and \eqref{eqn:GeoBounds}, we obtain
\begin{align}
\begin{split}
I_2 & = \big((E_\Gamma(\bu))^e, \Grad_{\Gamma_h}\bw\big[{\bf N}{\bf M} - {\bf P}_h\big]\big)_{\calT_h} \\
& \lesssim \big((E_\Gamma(\bu))^e, \Grad_{\Gamma_h}\bw\big[{\bf N} - {\bf P}_h\big]\big)_{\calT_h} + h^{k_g + 1}\|(E_\Gamma(\bu))^e\|_{L^2(\Gamma_h)}\|\bw\|_{H^1(\Gamma_h)} \\
& = - \bigg((E_\Gamma(\bu))^e, \frac{1}{\bn \cdot \bn_h}(\Grad_{\Gamma_h}\bw{\bf P}_h\bn) \otimes \bn_h\bigg)_{\calT_h} 
+ h^{k_g + 1}\|(E_\Gamma(\bu))^e\|_{L^2(\Gamma_h)}\|\bw\|_{H^1(\Gamma_h)} \\
& = \bigg((E_\Gamma(\bu))^e(\bn -\bn_h), \frac{1}{\bn \cdot \bn_h}\Grad_{\Gamma_h}\bw{\bf P}_h(\bn -\bn_h)\bigg)_{\calT_h} \\
& \qquad + h^{k_g + 1}\|(E_\Gamma(\bu))^e\|_{L^2(\Gamma_h)}\|\bw\|_{H^1(\Gamma_h)} \\
& \lesssim h^{k_g + 1}\|(E_\Gamma(\bu))^e\|_{L^2(\Gamma_h)}\|\bw\|_{H^1(\Gamma_h)}.
\end{split} \label{F.1-Estimate.II}
\end{align}

{\em Step (v), estimate of $I_3$}: By the Remark after \cite[Lemma 3.2]{LarssonLarson17}, for ${\bm \chi} \in [W_1^1(\Gamma_h)]^3$, there holds
\begin{equation} \label{eqn:ChiTrick}
|({\bm \chi}, {\bf P}_h \bn)_{\calT_h}| \lesssim 
h^{k_g + 1}\|{\bm \chi}\|_{W_1^1(\Gamma_h)}.
\end{equation}
Because $({\bf P} - {\bf P}_h)(\bn - \bn_h) = - {\bf P}\bn_h - {\bf P}_h\bn$, there  also holds
\begin{equation} \label{eqn:ChiTrick2}
|({\bm \chi}, {\bf P}\bn_h)_{\calT_h}| 
= |({\bm \chi}, {\bf P}_h\bn)_{\calT_h} + ({\bm \chi}, ({\bf P} - {\bf P}_h)(\bn - \bn_h))_{\calT_h}| \lesssim h^{k_g + 1}\|{\bm \chi}\|_{W_1^1(\Gamma_h)}.
\end{equation}
We then apply \eqref{eqn:GeoBounds} and \eqref{eqn:ChiTrick2} to obtain
\begin{align}
\begin{split}
I_3 & = - \big((E_\Gamma(\bu))^e, ({\bf L}\bn_h) \otimes ({\bf H}_h\bw){\bf N}{\bf M}\big)_{\calT_h} \\
& = - \big((E_\Gamma(\bu))^e, \mu_h^{- 1}({\bf P}\bn_h) \otimes ({\bf H}_h\bw){\bf N}{\bf M}\big)_{\calT_h} \\
& \qquad 
+ \big((E_\Gamma(\bu))^e, d\mu_h^{- 1}({\bf H}\bn_h) \otimes ({\bf H}_h\bw){\bf N}{\bf M}\big)_{\calT_h} \\
& = - \big(\mu_h^{- 1}(E_\Gamma(\bu))^e{\bf M}{\bf N}^\intercal{\bf H}_h\bw, {\bf P}\bn_h\big)_{\calT_h} \\
& \qquad + \big((E_\Gamma(\bu))^e, d\mu_h^{- 1}({\bf H}\bn_h) \otimes ({\bf H}_h\bw){\bf N}{\bf M}\big)_{\calT_h} \\
%
%
& \lesssim  h^{k_g + 1}\|(E_\Gamma(\bu))^e\|_{H^1(\Gamma_h)}\|\bw\|_{H^1(\Gamma_h)}.
\end{split} \label{F.1-Estimate.III}
\end{align}

{\em Step (vi), estimate of $I_4$}: Using \cite[(A.6)]{DemlowNeilan24},
\[
\nab{\bf L}\bw = - \mu_h^{- 1}(({\bf L}\bw) \otimes \nab\mu_h + \bn \otimes ({\bf H}\bw) + ({\bf H}\bw) \otimes \bn + (\bn \cdot \bw){\bf H} + d\nab{\bf H}\bw),
\]
which implies
\begin{align}
\begin{split}
& {\bf P}\nab{\bf L}\bw{\bf P}_h \\
& = - \mu_h^{- 1}(({\bf L}\bw) \otimes \nab\mu_h{\bf P}_h + ({\bf H}\bw) \otimes ({\bf P}_h\bn) + (\bw \cdot ({\bf P}_h \bn)){\bf H}{\bf P}_h + d{\bf P}\nab{\bf H}\bw{\bf P}_h), \label{F.20}
\end{split}
\end{align}
where we used ${\bf P}{\bf L} = {\bf L}$, ${\bf P}\bn = 0$, ${\bf P}{\bf H} = {\bf H}$ and $\bw = {\bf P}_h\bw$. Due to \eqref{F.20}, \eqref{eqn:GeoBounds}, the symmetry of $(E_\Gamma(\bu))^e$ and by \eqref{eqn:ChiTrick}, we obtain
\begin{align}
\begin{split}
|I_4| & = \Big|\big((E_\Gamma(\bu))^e, {\bf P}\nab{\bf L}\bw{\bf P}_h{\bf N}{\bf M}\big)_{\calT_h}\Big| \\
& \le \Big|\big(\mu_h^{- 1}{\bf P}_h(E_\Gamma(\bu))^e{\bf L}\bw, \nab\mu_h\big)_{\calT_h}\Big| 
+ \Big|\big(\mu_h^{- 1}(E_\Gamma(\bu))^e{\bf H}\bw, {\bf P}_h\bn\big)_{\calT_h}\Big| \\
& \qquad + \Big|\big(\mu_h^{- 1}(E_\Gamma(\bu))^e : ({\bf H}{\bf P}_h)\bw, {\bf P}_h\bn\big)_{\calT_h}\Big|
+ \Big|\big(\mu_h^{- 1}(E_\Gamma(\bu))^e, d{\bf P}\nab{\bf H}\bw{\bf P}_h\big)_{\calT_h}\Big| \\
& \lesssim \Big|\big(\mu_h^{- 1}{\bf P}_h(E_\Gamma(\bu))^e{\bf L}\bw, \nab\mu_h\big)_{\calT_h}\Big| + h^{k_g + 1}\|(E_\Gamma(\bu))^e\|_{H^1(\Gamma_h)}\|\bw\|_{H^1(\Gamma_h)}.
\end{split} \label{F.1-Estimate.IV}
\end{align}
We now estimate the remaining term in the right hand side of \eqref{F.1-Estimate.IV}.
Applying the product rule and Jacobi's formula (cf.~\cite[p.20]{DemlowNeilan24}),
we calculate
\begin{align*}
\nab\mu_h = {\bf H}\bn_h + {\bf H}_h\bn - {\rm tr}({\bf H})\bn + \mathcal{O}(h^{k_g + 1}). 
\end{align*}

Therefore by \eqref{eqn:ChiTrick},
\begin{align}
\begin{split}
& \Big|\big(\mu_h^{- 1}{\bf P}_h(E_\Gamma(\bu))^e{\bf L}\bw, \nab\mu_h\big)_{\calT_h}\Big| \\
& \lesssim \Big|\big((E_\Gamma(\bu))^e{\bf L}\bw, {\bf P}_h({\bf H}\bn_h + {\bf H}_h\bn)\big)_{\calT_h}\Big| \\
& \qquad + \Big|\big({\rm tr}({\bf H})(E_\Gamma(\bu))^e{\bf L}\bw, {\bf P}_h\bn\big)_{\calT_h}\Big| + h^{k_g + 1}\|(E_\Gamma(\bu))^e\|_{L^2(\Gamma_h)}\|\bw\|_{L^2(\Gamma_h)} \\
& \lesssim \Big|\big((E_\Gamma(\bu))^e{\bf L}\bw, {\bf P}_h({\bf H}\bn_h + {\bf H}_h\bn)\big)_{\calT_h}\Big| + h^{k_g + 1}\|(E_\Gamma(\bu))^e\|_{H^1(\Gamma_h)}\|\bw\|_{H^1(\Gamma_h)}.
\end{split} \label{F.21}
\end{align}
Note that ${\bf H}\bn_h = {\bf H}{\bf P}\bn_h$, and so by \eqref{eqn:ChiTrick2},
\begin{align}
\begin{split}
\Big|\big((E_\Gamma(\bu))^e{\bf L}\bw, {\bf P}_h{\bf H}\bn_h\big)_{\Gamma_h}\Big| & = \Big|\big({\bf H}{\bf P}_h(E_\Gamma(\bu))^e{\bf L}\bw, {\bf P}\bn_h\big)_{\Gamma_h}\Big| \\
& \lesssim h^{k_g + 1}\|(E_\Gamma(\bu))^e\|_{H^1(\Gamma_h)}\|\bw\|_{H^1(\Gamma_h)}.
\end{split} \label{F.22}
\end{align}
Likewise, ${\bf H}_h\bn = {\bf H}_h{\bf P}_h\bn$ and \eqref{eqn:ChiTrick} imply
\begin{align}
\begin{split}
\Big|\big((E_\Gamma(\bu))^e{\bf L}\bw, {\bf P}_h{\bf H}_h\bn\big)_{\calT_h}\Big| & = \Big|\big({\bf H}_h{\bf P}_h(E_\Gamma(\bu))^e{\bf L}\bw, {\bf P}_h\bn\big)_{\calT_h}\Big| \\
& \lesssim h^{k_g + 1}\|(E_\Gamma(\bu))^e\|_{H^1(\Gamma_h)}\|\bw\|_{H^1(\Gamma_h)}.
\end{split} \label{F.23}
\end{align}
It follows from \eqref{F.1-Estimate.IV}--\eqref{F.23} that
\begin{align}
|I_4| \lesssim h^{k_g + 1}\|(E_\Gamma(\bu))^e\|_{H^1(\Gamma_h)}\|\bw\|_{H^1(\Gamma_h)}. \label{F.1-Estimate.IV2}
\end{align}

{\em Step (vii), combining estimates}: Combining  the estimates \eqref{F.1-Estimate}, \eqref{F.1-Estimate.I}, \eqref{F.1-Estimate.II}, \eqref{F.1-Estimate.III} and \eqref{F.1-Estimate.IV2} yields
\begin{align}\label{eqn:LastStep}
\big((H_\Gamma(\psi))^e, (H_\Gamma(\chi))^e - H_{\Gamma_h}(\chi^e)\big)_{\calT_h} \lesssim h^{k_g + 1}\|(E_\Gamma(\bu))^e\|_{H^1(\Gamma_h)}\|\bw\|_{H^1(\Gamma_h)}.
\end{align}
By \eqref{F.Def} and \eqref{eqn:NormEquivy},
\begin{align*}
\|(E_\Gamma(\bu))^e\|_{H^1(\Gamma_h)} & = \|(H_\Gamma(\psi))^e\|_{H^1(\Gamma_h)} \lesssim \|H_\Gamma(\psi)\|_{H^1(\Gamma)} \lesssim \|\psi\|_{H^3(\Gamma)}, \\
\|\bw\|_{H^1(\Gamma_h)} & = \|\bcurl_{\Gamma_h}\chi^e\|_{H^1(\Gamma_h)} \lesssim \|\bcurl_\Gamma\chi\|_{H^1(\Gamma)} \lesssim \|\chi\|_{H^3(\Gamma)}.
\end{align*}
Inserting these inequalities into \eqref{eqn:LastStep} completes the proof of \eqref{F.1}.

\item \textit{Estimate} \eqref{F.2}: By \eqref{eqn:gradChain},
\[
(\nab_\Gamma\chi)^e - \nab_{\Gamma_h}\chi^e = ({\bf I} - {\bf P}_h({\bf P} - d{\bf H}))(\nab_\Gamma\chi)^e = ({\bf I} - {\bf P}_h{\bf P} + d{\bf P}_h{\bf H})(\nab_\Gamma\chi)^e.
\]
Note that
\begin{align*}
{\bf I} - {\bf P}_h{\bf P} & = {\bf I} - ({\bf I} - \bn_h \otimes \bn_h)({\bf I} - \bn \otimes \bn) \\
& = \bn \otimes \bn + \bn_h \otimes \bn_h - (\bn \cdot \bn_h)\bn_h \otimes \bn \\
& = (\bn - \bn_h) \otimes (\bn - \bn_h) + (1 - \bn \cdot \bn_h)\bn_h \otimes \bn + \bn \otimes \bn_h,
\end{align*}
and
\[
\big((\nab_\Gamma\psi)^e, \bn \otimes \bn_h(\nab_\Gamma\chi)^e\big)_{\Gamma_h} = 0.
\]
Then, by a change of integration domain, \eqref{eqn:GeoBounds} and \eqref{eqn:one-nnh},
\begin{align*}
\big((\nab_\Gamma\psi)^e, (\nab_\Gamma\chi)^e - \nab_{\Gamma_h}\chi^e\big)_{\Gamma_h} & \lesssim h^{2k_g}\|\psi\|_{H^1(\Gamma)}\|\chi\|_{H^1(\Gamma)} + h^{k_g + 1}\|\psi\|_{H^1(\Gamma)}\|\chi\|_{H^1(\Gamma)} \\
& \lesssim h^{k_g + 1}\|\psi\|_{H^1(\Gamma)}\|\chi\|_{H^1(\Gamma)}.
\end{align*}

\item \textit{Estimate} \eqref{F.3}: 
As $\psi, \chi \in H^4(\Gamma)$, both
$H_\Gamma(\psi)$ and $\nab_\Gamma\chi$ are continuous, which imply $\avg{\bt_{T^{\ell}}^{\intercal}H_{\Gamma}(\psi)\bmu_{T^{\ell}}} = \bt_{T^{\ell}}^{\intercal}H_{\Gamma}(\psi)\bmu_{T^{\ell}}$ and $(\jump{\nab_{\Gamma}\chi})^e = 0$. Thus, by \eqref{eqn:gradChain},
\begin{align}
\begin{split}
& \big((\avg{\bt_{T^{\ell}}^{\intercal}H_{\Gamma}(\psi)\bmu_{T^{\ell}}})^{e}, (\jump{\nab_{\Gamma}\chi})^{e} - \jump{\nab_{\Gamma_{h}}\chi^{e}}\big)_{\calE_h} \\
& = - \big((\bt_{T^{\ell}}^{\intercal}H_{\Gamma}(\psi)\bmu_{T^{\ell}})^{e}, \jump{{\bf P}_h({\bf P} - d{\bf H})(\nab_\Gamma\chi)^e}\big)_{\calE_h} \\
& = \big((\bt_{T^{\ell}}^{\intercal}H_{\Gamma}(\psi)\bmu_{T^{\ell}})^{e}, \jump{d{\bf P}_h{\bf H}(\nab_\Gamma\chi)^e}\big)_{\calE_h} - \big((\bt_{T^{\ell}}^{\intercal}H_{\Gamma}(\psi)\bmu_{T^{\ell}})^{e}, \jump{{\bf P}_h(\nab_\Gamma\chi)^e}\big)_{\calE_h} \\
& =: I + II.
\end{split} \label{F.3-Estimate}
\end{align}
First, we estimate $I$ in \eqref{F.3-Estimate}. Using
\[
{\bf P}(\bmu_+ + \bmu_-) = {\bf P}\bmu_+ - \bmu_{T^\ell+}^e + {\bf P}\bmu_- - \bmu_{T^\ell-}^e,
\]
and \eqref{GeoApprox.3}, \eqref{eqn:GeoBounds}, \eqref{tracineq} and a change of integration domain, there holds
\begin{align}
\begin{split}
I 
& = \big((\bt_{T^{\ell}}^{\intercal}H_{\Gamma}(\psi)\bmu_{T^{\ell}})^{e}, ({\bf P}(\bmu_+ + \bmu_-)) \cdot (d{\bf H}(\nab_\Gamma\chi\big)^e))_{\calE_h} \\
& \lesssim h^{k_g+1} \Big(\|{\bf P}\bmu_+ - \bmu_{T^\ell+}^e\|_{L^\infty(\calE_h)} + \|{\bf P}\bmu_- - \bmu_{T^\ell-}^e\|_{L^\infty(\calE_h)}\Big) 
\|(\nab_\Gamma^2\psi)^e\|_{L^2(\calE_h)}\|(\nab_\Gamma\chi)^e\|_{L^2(\calE_h)} \\
& \lesssim h^{2k_g + 1}\|\psi\|_{H^3(\Gamma)}\|\chi\|_{H^2(\Gamma)},
\end{split} \label{F.3-Estimate.I}
\end{align}
where we also used $\bmu \cdot \bn_h = 0$, ${\bf P}^\intercal = {\bf P}$ and ${\bf PH} = {\bf P}$. 

Next, we estimate $II$ in \eqref{F.3-Estimate}. 
We make change of variables to obtain
\begin{align*}
II
& = - \big((\bt_{T^{\ell}}^{\intercal}H_{\Gamma}(\psi)\bmu_{T^{\ell}})^{e}, \jump{{\bf P}_h(\nab_\Gamma\chi)^e}\big)_{\calE_h} \\
& = 
-\sum_{e^\ell\in \calE^\ell_h} \left(\mu_e^{-1} (\bt^\intercal_{e^\ell} H_\Gamma(\psi) \bmu_{e^\ell}),({\bf P}\bmu^\ell_++{\bf P}\bmu_-^\ell)\cdot \nab_\Gamma \chi\right)_{e^\ell}.
\end{align*}
Recall  $\bmu_{\pm} = \bt_{\pm} \times \bn_{\pm}$, we have
\begin{align*}
{\bf P}\bmu^\ell_{\pm} 
&= {\bf P}(\bt_{{\pm}}^\ell \times \bn_{{\pm}}^\ell)
 = (\bn\cdot \bn_{\pm}^\ell)(\bt_{\pm}^\ell\times \bn) + (\bn\cdot \bt^\ell_{\pm})(\bn\times \bn_{\pm}^\ell).
\end{align*}
Consequently, since $\bt_+ = -\bt_-$, there holds
\begin{align*}
{\bf P}\bmu^\ell_+ + 
{\bf P}\bmu^\ell_{-}
& = 
  (\bn\cdot \bn_{+}^\ell)(\bt_{+}^\ell\times \bn) + (\bn\cdot \bt^\ell_{+})(\bn\times \bn_{+}^\ell)
 -  (\bn\cdot \bn_{-}^\ell)(\bt_{+}^\ell\times \bn) - (\bn\cdot \bt^\ell_{+})(\bn\times \bn_{-}^\ell)\\
& = (\bt_+^\ell \times \bn)\left(\bn\cdot \bn_+^\ell - \bn\cdot \bn_-^\ell\right)
+(\bn\cdot \bt^\ell_+)\left(\bn\times \bn_+^\ell - \bn\times \bn_-^\ell\right)\\
& = (\bt_+^\ell \times \bn)\left(\bn\cdot \bn_+^\ell-1 - \left((\bn\cdot \bn_-^\ell)-1\right)\right)\\
&\qquad+\left(\bn\cdot (\bt^\ell_+-\bt)\right)\left(\bn\times (\bn_+^\ell-\bn) - \bn\times (\bn_-^\ell-\bn)\right),
 \end{align*}
and so $\left|{\bf P}\bmu^\ell_+ + 
{\bf P}\bmu^\ell_{-}\right|\lesssim h^{2k_g}$.
 We then apply trace inequalities to conclude
 \begin{align}\label{eqn:Bleh}
II
\lesssim h^{2k_g-1} \|\psi\|_{H^3(\Gamma)} \|\chi\|_{H^2(\Gamma)}.
\end{align}

Finally, the result follows from \eqref{F.3-Estimate}--\eqref{eqn:Bleh}.
\end{itemize}
\end{proof}

\section{Proof of Lemma \ref{thm:RegBihar}}\label{app-Regularity}

To prove the regularity estimates, we require an intermediate result.
\begin{lemma}\label{lem:PrelimSur}
Let $m\in \{0,1\}$. For any $\chi\in H^{-m}(\Gamma)$, there exists $\br\in \bH^{1-m}(\Gamma)$ such that
${\rm curl}_\Gamma \br = -\chi$ and $\|\br\|_{{\bH}^{1-m}(\Gamma)}\lesssim \|\chi\|_{H^{-m}(\Gamma)}$.
\end{lemma}
\begin{proof}
Let $\psi\in H^1(\Gamma)$ satisfy $-\Delta_\Gamma \psi = \chi$. Elliptic regularity yields
$\psi\in H^{2-m}(\Gamma)$ and $\|\psi\|_{H^{2-m}(\Gamma)}\lesssim \|\chi\|_{H^{-m}(\Gamma)}$ \cite[Lemma 3.2]{DziukElliott13}.
We then set $\br = \bcurl_\Gamma  \psi\in \bH^{1-m}(\Gamma)$, so that
\[
{\rm curl}_\Gamma \br = {\rm curl}_\Gamma \bcurl_\Gamma \psi = \Delta_\Gamma \psi = - \chi,
\]
and $\|\br\|_{H^{1-m}(\Gamma)}\lesssim \|\psi\|_{H^{2-m}(\Gamma)}\lesssim \|\chi\|_{H^{-m}(\Gamma)}$.
\end{proof}

\begin{proof}[Proof of Lemma \ref{thm:RegBihar}]
Using Lemma \ref{lem:PrelimSur}, we let $\br\in \bH^{1-m}(\Gamma)$ satisfy ${\rm curl}_\Gamma \br = -\chi$. 
Setting $\bv = \bcurl_\Gamma u$,
we then have
\begin{align*}
-{\bf P}{\rm div}_\Gamma E_\Gamma(\bv) + \nab_{\Gamma} s & = \br,\\
{\rm div}_\Gamma \bv & = 0,
\end{align*}
for some $s\in L^2(\Gamma)$.
Using the elliptic regularity of the surfaces Stokes problem $\bv\in H^{3-m}(\Gamma)$ with $\|\bv\|_{H^{3-m}(\Gamma)}\lesssim \|\br\|_{H^{1-m}(\Gamma)}$,
along with the Poincare inequality \cite[Theorem 2.12]{DziukElliott13},
we have
\begin{align*}
\|u\|_{H^{4-m}(\Gamma)}\lesssim \|\nab_\Gamma u\|_{H^{3-m}(\Gamma)}\lesssim \|\bn\times \bv\|_{H^{3-m}(\Gamma)}\lesssim \|\br\|_{H^{1-m}(\Gamma)}\lesssim \|\chi\|_{H^{-m}(\Gamma)}.
\end{align*}
\end{proof}

\end{document}